\documentclass[a4paper,12pt,reqno,oneside]{amsart}
\usepackage{amssymb}
\usepackage{graphicx,color}
\usepackage{newtxtext}
\usepackage[varg]{newtxmath}
\usepackage[shortlabels]{enumitem}
\usepackage{mathtools}
\mathtoolsset{showonlyrefs=true}
\usepackage{geometry}%
\newtheorem{theorem}{Theorem}[section]
\newtheorem{corollary}[theorem]{Corollary}
\newtheorem{proposition}[theorem]{Proposition}
\newtheorem{lemma}[theorem]{Lemma}
\theoremstyle{definition} 

\newtheorem{remark}[theorem]{Remark}

\DeclareMathOperator{\Ker}{Ker}

\newcommand{\N}{\mathbb{N}}
\newcommand{\Z}{\mathbb{Z}}

\newcommand{\R}{\mathbb{R}}

\usepackage{constants}
\newconstantfamily{eps}{symbol=\varepsilon}
\usepackage[mathlines]{lineno}

\usepackage{amsrefs}
\usepackage{bm}

\usepackage{hyperref}
\newcommand{\T}{\mathbb{T}}

\title{The \L{}ojasiewicz-Simon inequality related to grain boundary
motion and its applications}

\author{Masashi Mizuno}
\address[Masashi Mizuno]%
{
Department of Mathematics, College of Science
and Technology, Nihon University, Tokyo 101-8308 Japan \newline
Division of Joint Research of Geometry and Natural Science, Research Institute for Science and Technology, Tokyo University of Science,
Chiba, 278-8510, Japan
}
\email{mizuno.masashi@nihon-u.ac.jp}

\author{Ayumi Sakiyama}
\address[Ayumi Sakiyama]%
{Nissay Information Technology Co., Ltd.}

\author{Keisuke Takasao}
\address[Keisuke Takasao]%
{Department of Mathematics, Graduate School of Science, Kyoto University, Kitashirakawa-Oiwakecho Sakyo Kyoto 606-8502, Japan}
\email{k.takasao@math.kyoto-u.ac.jp}

\keywords{\L{}ojasiewicz-Simon inequality, grain boundary motion}
\subjclass[2020]{Primary: 35K93, Secondary: 53E10, 70G75, 74N15}

\allowdisplaybreaks[1]
\begin{document}

\begin{abstract}
In this paper, we study the \L{}ojasiewicz-Simon gradient inequality for
the mathematical model of grain boundary motion. We first derive a curve
shortening equation with time-dependent mobility, which guarantees the
energy dissipation law for the grain boundary energy, including the
 difference between orientations of the constituent grains as a state
variable. Next, we discuss the \L{}ojasiewicz-Simon gradient inequality
for the grain boundary energy. Finally, we give applications of the
inequality to the energy.
\end{abstract}

\maketitle

\section{Introduction}
\label{sec:Introduction}
Let $\mathbb{T}=\R/\Z$ be a torus. We consider the following energy
functional:
\begin{equation}
 \label{eq:1.GBE}
 E[u,\alpha]
  =
  \sigma(\alpha)\int_{0}^{1}\sqrt{1+|u_{x}(x)|^{2}}dx.
\end{equation}
Here, $\alpha\in\R$ is a constant, $u:\mathbb{T}\rightarrow\R$ is a
function on $\T$, and $\sigma:\R\rightarrow\R$ is a function on $\R$.
Throughout this paper, we assume that there is a positive constant
$\Cl{const:GBE_positivity}>0$ such that
\begin{equation}
 \label{eq:1.Assumption_positivity}
 \sigma(\alpha)
  \geq
  \Cr{const:GBE_positivity}
\end{equation}
for $\alpha\in\R$.

The energy functional \eqref{eq:1.GBE} is related to the mathematical
model of grain boundary motion. \cite{MR4263432} proposed a mathematical
model of grain boundary motion that includes the evolution of
misorientations and triple junction mobility. Here, misorientation of a
grain boundary is the difference between the orientations of two grains
constituting the grain boundary. In this model, the normal velocity of
the grain boundary depends not only on its curvature but also on the
misorientation of its grain boundary. The energy dissipation by the
triple junction determines the angle at the edge point of the three
grain boundaries. Note that in \cite{MR1833000}, the misorientation
effect was included in the grain boundary energy as constant parameters.

Curve shortening flow, or mean curvature flow, is well-known as a
mathematical model of grain boundary motion (\cites{MR1770892,
MR0078836, doi:10.1063/1.1722742}). Although grain boundaries are formed
by the misorientation and orientation mismatch, the effect of the
misorientation on the grain boundary motion is not clear. To understand
the interaction of the misorientation with the grain boundary motion, we
derive a mathematical model for one grain boundary moved by its
curvature and misorientation. Let $u=u(x,t)$ be a representation of the
grain boundary (i.e., we regard $\{(x,u(x,t)):0\leq x\leq 1\}$ as the
grain boundary) and $\alpha=\alpha(t)$ be a misorientation on the grain
boundary. Similar to \cite{MR4263432}, we compute the energy dissipation
rate as
\begin{equation*}
 \begin{split}
  \frac{d}{dt}E[u,\alpha]
  &=
  \sigma'(\alpha(t))
  \alpha_t(t)
  \int_0^1
  \sqrt{1+(u_x(x,t))^2}
  \,dx
  \\
  &\quad
  +
  \sigma(\alpha(t))
  \int_0^1
  \frac{u_x(x,t)}{\sqrt{1+(u_x(x,t))^2}}u_{xt}(x,t)
  \,dx \\
  &=
  \sigma'(\alpha(t))
  \alpha_t(t)
  \int_0^1
  \sqrt{1+(u_x(x,t))^2}
  \,dx
  \\
  &\quad
  -
  \sigma(\alpha(t))
  \int_0^1
  \left(
  \frac{u_x(x,t)}{\sqrt{1+(u_x(x,t))^2}}
  \right)_x
  u_{t}(x,t)
  \,dx \\
  &\quad
  +
  \sigma(\alpha(t))
  \left(
  \frac{u_x(1,t)}{\sqrt{1+(u_x(1,t))^2}}
  u_{t}(1,t)
  -
  \frac{u_x(0,t)}{\sqrt{1+(u_x(0,t))^2}}
  u_{t}(0,t)
  \right).
 \end{split}
\end{equation*}
In order to ignore the triple junction effect, we assume the periodic
boundary condition at the grain boundary $u$. Then, we obtain
\begin{equation*}
  \begin{split}
  \frac{d}{dt}E[u,\alpha]
  &=
  \sigma'(\alpha(t))
  \alpha_t(t)
  \int_0^1
  \sqrt{1+(u_x(x,t))^2}
  \,dx
  \\
  &\quad
  -
  \sigma(\alpha(t))
  \int_0^1
  \left(
   \frac{u_x(x,t)}{\sqrt{1+(u_x(x,t))^2}}
  \right)_x
  u_{t}(x,t)
  \,dx.
  \end{split}
\end{equation*}
To guarantee the energy dissipation law of the following form with
positive constants $\mu$, $\gamma>0$
\begin{equation}
 \label{eq:1.EnergyLaw}
 \begin{split}
  \frac{d}{dt}E[u,\alpha](t)
  =
  -\frac{1}{\gamma}|\alpha_t(t)|^{2}
  -
  \frac{1}{\mu}
  \int_0^1
  \left|
  \frac{u_{t}(x,t)}{\sqrt{1+u_{x}^{2}(x,t)}}
  \right|^2
  \sqrt{1+u_{x}^{2}(x,t)}dx,
 \end{split}
\end{equation}
we obtain the system of differential equations
\begin{equation}
 \label{eq:1.GeomEvolEq}
  \left\{
   \begin{aligned}
    \frac{u_{t}(x,t)}{\sqrt{1+u_{x}^{2}(x,t)}}
    &=
    \mu\sigma(\alpha(t))\left(\frac{u_{x}(x,t)}{\sqrt{1+u_{x}^{2}(x,t)}}\right)_{x},
    \\
    \alpha_t(t)
    &=-\gamma\sigma'(\alpha(t))\int_{0}^{1}\sqrt{1+u_{x}^{2}(x,t)}dx.
   \end{aligned}
  \right.
\end{equation}
In other words, \eqref{eq:1.GeomEvolEq} is a sufficient condition to
obtain \eqref{eq:1.EnergyLaw} with the periodic boundary condition. Note
that we may obtain other systems if we consider a different type of energy
law or energy functional.

We are interested in a sufficient condition for obtaining the
\L{}ojasiewicz-Simon gradient inequality and analyzing long-time
behavior of solutions of \eqref{eq:1.GeomEvolEq}. If $\sigma$ has the
following convexity condition
\begin{equation}
 \label{eq:1.ConvexityAssumption}
 \alpha\sigma'(\alpha)\geq0,\qquad \alpha\in\R,
\end{equation}
then long-time behavior of solutions to \eqref{eq:1.GeomEvolEq} was
studied by \cite{MR4292952}. The convexity assumption
\eqref{eq:1.ConvexityAssumption} guarantees the maximum principle for
$\alpha$. However, from the viewpoint of grain boundary motion, $\sigma$
should be periodic with respect to $\alpha$. For instance, as shown in
Figure \ref{fig:1.GB-Orientation}, consider a grain boundary composed of
two grains with different lattice angles.  Even if one of the grains
rotates by an angle of $\frac{\pi}2$, the structure of the grain
boundary does not change. This means that the grain boundary energy
density $\sigma$ should be periodic with respect to $\alpha$, so we may
not assume the convexity assumption \eqref{eq:1.ConvexityAssumption}(See
also \cite{yang2025curvatureflownetworkstriple}). Furthermore, as
discussed in \cite{MR4263432, MR4283537}, a typical example for grain
boundary energy density $\sigma$ is trigonometric functions. Thus, we
generally do not assume the convexity condition
\eqref{eq:1.ConvexityAssumption} from the viewpoint of applications to
grain boundary motion. Therefore, we want to study long-time asymptotics
for solutions of \eqref{eq:1.GeomEvolEq} without assuming the convexity
condition \eqref{eq:1.ConvexityAssumption}. In this paper, we study
long-time behavior of \eqref{eq:1.GeomEvolEq} by establishing the
\L{}ojasiewicz-Simon gradient inequality for the grain boundary energy
\eqref{eq:1.GBE}.

We note that, from the mathematical model of grain boundary motion,
$\alpha=0$ means that the grain boundary is not formed and the grain
boundary energy should vanish. Thus, it is natural to assume that
$\sigma(0)=0$. The main results of this study cannot address this
consideration, but when \eqref{eq:1.Assumption_positivity} is not
assumed, then exponential decay of the solution cannot be expected as in
\cite{MR4292952} in general. Therefore, we consider the
\L{}ojasiewicz-Simon inequality that can handle the algebraic decay of
solutions for a broad class of partial differential equations.
On the other hand, as suggested in \cite[(2.10)]{MR4292952}, the grain
boundary energy density $\sigma$ may be depending not only on the
misoerientation $\alpha$ but also depending on the unit normal vector of
the graph of $u$. Furthermore, \cite[(2.17)]{MR4263432} suggested the
triple junction drag, which is essentially same as the dynamic boundary
problem to \eqref{eq:1.GeomEvolEq}.
In the theory of partial differential equations, the case of higher
space dimensions and other boundary conditions, such as the Dirichlet or
the Neumann boundary condition, is an interesting extension of the
problem \eqref{eq:1.GeomEvolEq}. In this paper, as a first step, we
assume that $\sigma$ depends only on $\alpha$ and satisfies
\eqref{eq:1.Assumption_positivity}, find suitable function spaces
containing periodic functions to derive the \L{}ojasiewicz-Simon
inequality for the energy \eqref{eq:1.GBE}, and study its applications.

\begin{figure}
 \includegraphics[bb=0 0 499 163,width=.9\textwidth]{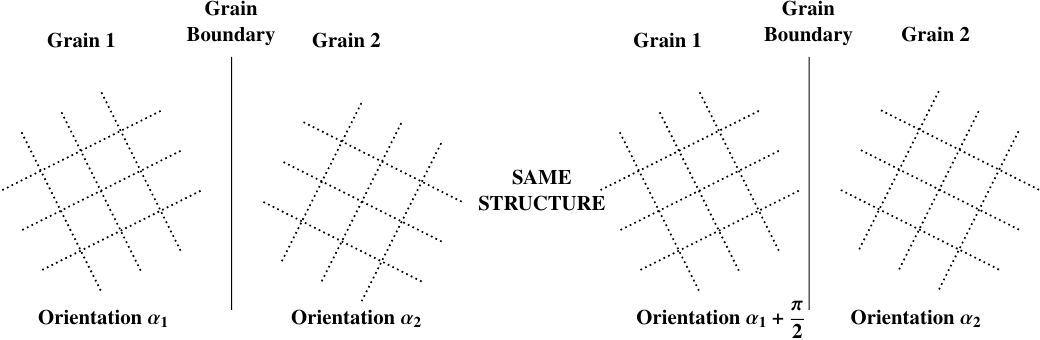}
 \label{fig:1.GB-Orientation}
 \caption{The left figure is a schematic of two grains (Grain 1, Grain
2) and a single grain boundary. The right figure is the one with
$\frac{\pi}{2}$ rotation for the crystal lattice of Grain 1. As can be
seen, the structure of the grain boundary does not change by
$\frac{\pi}{2}$ rotation, so the grain boundary energy density $\sigma$
should be periodic in misorientation.  }
\end{figure}

The organization of this paper is as follows. In section \ref{sec:LS},
we introduce the \L{}ojasiewicz-Simon inequality on abstract function
spaces by Yagi~\cite{MR4274456}, Chill~\cite{MR1986700}, and apply their
results to the grain boundary energy~\eqref{eq:1.GBE}. Applications of
the \L{}ojasiewicz-Simon inequality to the evolution of many grains can
be found in \cite{MR3501845, MR4753074}(see also \cite{MR4129083} in the case of single grain boundary
represented by a graph). In this paper, we provide full
details for deriving the \L{}ojasiewicz-Simon inequality for a single
grain boundary. In section \ref{sec:Application}, we analyze long-time
asymptotic behavior of solutions of the system \eqref{eq:1.GeomEvolEq}
by applying the \L{}ojasiewicz-Simon inequality. Further, we explain the
relationships among grain boundary length, curvature, and
misorientation.

\section{The \L ojasiewicz-Simon inequality}
\label{sec:LS}

In this section, we will derive the \L ojasiewicz-Simon inequality for
the grain boundary energy functional \eqref{eq:1.GBE}. To derive the
inequality, we first present Yagi's abstract theory of the
\L{}ojasiewicz-Simon inequality \cite{MR4274456}. Next, we investigate a
suitable function space for the grain boundary energy \eqref{eq:1.GBE} in order
to apply the abstract theory. Finally, we state and give a proof of the
\L ojasiewicz-Simon inequality to the functional \eqref{eq:1.GBE}.

\subsection{Abstract theory of the \L ojasiewicz-Simon inequality}
In \cite{MR4274456}, Yagi gave an abstract theory of the
\L{}ojasiewicz-Simon inequality. Here, we present the sufficient
condition to derive the \L ojasiewicz-Simon inequality in abstract
function spaces.

To state the theorem, we give some notation. Let $X$ be a real Hilbert
space with inner product $(\cdot, \cdot)_X$. The function
$\Phi:X\rightarrow\R$ is continuously differentiable if for each $f\in
X$, there exists $\dot{\Phi}(f)\in X$ such that
\begin{equation*}
 \Phi(f+h)
  -
  \Phi(f)
  -
  (\dot{\Phi}(f),h)_X
  =
  o(\|h\|_X),
  \quad
  \text{as}
  \
  h\rightarrow0\
  \text{in}\ X
\end{equation*}
and
\begin{equation*}
 X\ni f\mapsto
  \dot{\Phi}(f)
\in X\
\text{is continuous from}\
X\
\text{into itself.}
\end{equation*}
The vector $\dot{\Phi}(f)$ is called the Fr\'{e}chet derivative of $\Phi$
at $f$. A vector $\overline{u}\in X$ is called a critical point of
$\Phi$ if $\dot{\Phi}(\overline{u})=0$.

More generally, for two real Banach spaces $X$, $Y$, and for a map $F$
from $X$ to $Y$, we call $F$ Fr\'{e}chet differentiable at $u\in X$ if
there exists bounded linear operator $F'(u):X\rightarrow Y$ such that
\begin{equation*}
 F(u+h)
  -
  F(u)
  -
  F'(u)h
  =
  o(\|h\|_X)
  \quad
  \text{as}\
  h\rightarrow0\
  \text{in}\
  X.
\end{equation*}
The map $F$ is called G\^{a}teaux differentiable at $u\in X$ if
$\R\ni\theta\mapsto F(u+\theta h)\in Y$ is differentiable at $\theta=0$
for any $h\in X$. We denote
\begin{equation*}
 DF(u,h):=\frac{d}{d\theta}\bigg|_{\theta=0}F(u+\theta h)
\end{equation*}
and call $DF(u,\cdot)$ the G\^{a}teaux derivative of $F$ at $u$. A
bounded linear operator $T: X\rightarrow Y$ is called a Fredholm
operator if its kernel is a subspace of $X$ of finite dimension and its
range is a closed subspace of $Y$ of finite codimension.
See \cite{MR0223930} about the Fredholm operator, for instance.

\begin{theorem}[Theorem 3.1 in {\cite{MR4274456}}]
 \label{thm:LS.LS} Let $X$ be a real Hilbert space. Let
 $\Phi:X\rightarrow\R$ be a continuously differentiable function with
 the derivative $\dot{\Phi}:X\rightarrow X$, and let $\bar{u}$ be its
 critical point. Assume that $u\mapsto\dot{\Phi}(u)$ is G\^{a}teaux
 differentiable at $\bar{u}$ with derivative
 $L=D\dot{\Phi}(\bar{u},\cdot)$ and $L$ is a Fredholm operator of
 $X$. Assume also that there exists a Banach space $Y$ satisfying the
 following five conditions:
 \begin{equation}
  \label{eq:LS.LS.Abstract_Assumption}
  \begin{aligned}
   & Y\subset X\
   \text{with dense and continuous embedding;}
   \\
   & \bar{u}\in Y;
   \\
   &L^{-1}(Y)\subset Y,\
   \text{namely, if}\
   Lu\in Y,\ \text{then}\ u\in Y;
   \\
   &\dot{\Phi}\
   \text{maps}\ Y\
   \text{into itself;}
   \\
   &\dot{\Phi}:Y\rightarrow Y\
   \text{is continuously Fr\'{e}chet differentiable in}
   \ Y.
  \end{aligned}
 \end{equation}
 Define the critical
 manifold $S$ as
 \begin{equation*}
  S:=
   \{u\in Y:L(\dot{\Phi}(u))=0\}
 \end{equation*}
 and assume that $\Phi$ is analytic on $S$. Then, in a
 neighborhood $U$ of $\bar{u}$ in $Y$, $\Phi(u)$ satisfies the gradient
 inequality
 \begin{equation*}
  \|\dot{\Phi}(u)\|_Y
   \geq
   \Cr{const:LS_abstract}
   |\Phi(u)-\Phi(\bar{u})|^{1-\theta}
 \end{equation*}
 for all $u\in U$, where $0<\theta\leq\frac{1}{2}$ and
 $\Cl{const:LS_abstract}>0$ is a constant.
\end{theorem}

We note the analyticity of $\Phi$ on the critical manifold $S$. We refer to
\cite[{\S 3.3}]{MR4274456} for a precise explanation. Let $P$ be the
orthogonal projection from $X$ to $\Ker L$. By the implicit function
theorem, for any $\overline{u}\in S$, there are neighborhoods $U_0$,
$U_1$ of $P\overline{u}$, $(I-P)\overline{u}$ in $\Ker L$, $L(Y)$
respectively and a map $g:U_0\rightarrow U_1$ such that
\begin{equation*}
 S\cap U
  =
  \{u_0+g(u_0):u_0\in U_0\}.
\end{equation*}
Since $L$ is a Fredholm operator, $\Ker L$ is a finite-dimensional
subspace, thus we can take a basis $v_1,\ldots,v_N$ of $\Ker L$. Then,
we can identify $\Ker L$ with $\R^N$ by the correspondence
\begin{equation*}
 u_0=\sum_{k=1}^N\xi_kv_k \in \Ker L
  \longleftrightarrow
  (\xi_1,\ldots,\xi_N)\in\R^N.
\end{equation*}
Then $\Phi$ is analytic on $S\cap U$ if for any $\overline{u}\in S$,
\begin{equation*}
 (\xi_1,\ldots,\xi_N)\in\Omega
  \mapsto
  \Phi
  \left(
  \sum_{k=1}^N\xi_kv_k
  +
  g\left(
   \sum_{k=1}^N\xi_kv_k
   \right)
  \right)
\end{equation*}
is analytic, where $\Omega\subset\R^N$ is an open neighborhood
corresponding to $U_0$.

\subsection{Function spaces and the \L ojasiewicz-Simon inequality}

We next investigate the function space in order to apply the abstract
theory of the \L ojasiewicz-Simon inequality to the energy functional
$E$.

Let $X_u$ be a Hilbert space for the function $u$ of the energy $E$. We
will find the suitable domain $X=X_u\times\R$ of the grain boundary
energy $E$. Since $E$ is the integral of $u_x$, $X_u$ should be a subset
of the Sobolev space $H^1(0,1)$. In addition, since $u$ is a function of
$\T$, we should impose $u_x\in L^2(\T)$.

The first equation of \eqref{eq:1.GeomEvolEq} can be turned into a
divergence form, namely
\begin{equation*}
 u_{t}(x,t)
  =
  \mu\sigma(\alpha(t))(\arctan(u_{x}(x,t)))_{x}.
\end{equation*}
By integrating with respect to $x\in(0,1)$, we have for any solution
$(u,\alpha)$ of \eqref{eq:1.GeomEvolEq}
\begin{equation*}
 \frac{d}{dt}
  \int_0^1
  u(x,t)
  \,dx
  =
  \mu\sigma(\alpha(t))
  \int_0^1
  (\arctan(u_{x}(x,t)))_{x}
  \,dx
  =0,
\end{equation*}
which gives the conservation of the integral of $u$ along the flow
\eqref{eq:1.GeomEvolEq}. Adding a constant to the initial data implies
the parallel transform along the $y$ axis, and the shape of the curve
$\{(x,u(x,t)):x\in(0,1)\}$ does not change. Thus, we define the
following function spaces for $k=1,2$,
\begin{equation*}
 H_{\mathrm{per.ave}}^{k}(0,1)
 :=
 \left\{u\in H^{k}(0,1) 
 :
 u_{x}\in H^{k-1}(\T),\int_{0}^{1}u(x)dx=0\right\}.
\end{equation*}

\begin{remark}
 We emphasize that we \emph{do not} need to impose $u\in L^2(\T)$ since
 $u$ itself does not appear in the definition of $E$. In other words, we
 study the \L{}ojasiewicz-Simon inequality on a wider class than the
 class of the original functional \eqref{eq:1.GBE}. Note that even if
 $u_x$ is a periodic function on $[0,1)$, the primitive $u$ is not a
 periodic function on $[0,1)$ in general. For instance, let $u_x(x)=1$
 for $x\in[0,1)$. Then $u(x)$ can be written explicitly as
 $u(x)=u(0)+x$, which is not a periodic function on $x\in [0,1)$.

 Furthermore, to apply the problem \eqref{eq:1.GeomEvolEq}, we can
 deduce spatial periodicity of solutions $u$ of \eqref{eq:1.GeomEvolEq}
 if $u(x,0)$ and $u_x(x,t)$ are periodic functions with respect to $x$
 for all $t>0$. To prove this, we note that
 \begin{equation}
  \label{eq:LS.Integral_representation}
   u(x,t)
   =
   u(x,0)
   +
   \int_0^t
   u_t(x,\tau)\,d\tau
   =
   u_0(x)
   +
   \mu\sigma(\alpha(t))
   \int_0^t
   (\arctan(u_{x}(x,\tau)))_{x}
   \,d\tau.
 \end{equation}
 If $u_0(x)$ and $u_x(x,t)$ are spatially periodic functions on $[0,1)$,
 then the right-hand side of \eqref{eq:LS.Integral_representation} is
 also spatially periodic, hence $u(x,t)$ is also spatially periodic for all
 $t>0$.
\end{remark}

Next, we define Hilbert spaces $X$ and $Y$ as
 \begin{equation}
  \label{eq:LS.Function_Spaces}
  X:=H_{\mathrm{per.ave}}^{1}(0,1)\times\R,\quad
  Y:=H_{\mathrm{per.ave}}^{2}(0,1)\times\R.
 \end{equation}
For $(u_{1},\alpha_{1})$, $(u_{2},\alpha_{2})\in X$, and $\gamma\in\R$,
define the addition and scalar multiplication on $X$ by
\begin{equation*}
 (u_{1},\alpha_{1})+(u_{2},\alpha_{2})
 =
 (u_{1}+u_{2},\alpha_{1}+\alpha_{2}),\qquad
 \gamma(u_{1},\alpha_{1})
  =
  (\gamma u_{1},\gamma\alpha_{1}).
\end{equation*}
Since we can apply the Poincar\'e inequality in $X$, we may define the
inner product on $X$ as
\begin{equation}
 \label{eq:LS.def_InnerProduct}
 ((u,\alpha),(v,\beta))_{X}
 =\int_{0}^{1}u_{x}(x)v_{x}(x)dx+\alpha\beta
\end{equation}
for $(u,\alpha)$, $(v,\beta)\in X$. Similarly, we define the inner product on $Y$ for $(u,\alpha)$, $(v,\beta)\in Y$ as
\begin{equation*}
 \label{eq:LS.def_InnerProduct_Y}
 ((u,\alpha),(v,\beta))_{Y}
 =
 \int_{0}^{1}u_{x}(x)v_{x}(x)dx
 +
 \int_{0}^{1}u_{xx}(x)v_{xx}(x)dx
 +
 \alpha\beta.
\end{equation*}

\begin{remark}
 It is easy to show that for any $f\in H^{1}(0,1)$
 \begin{equation}
  \label{eq:LS.Sobolev_Embedding1}
  \sup_{x\in(0,1)}\left|f(x)-\int_{0}^{1}f(y)dy\right|
  \leq\|f_{x}\|_{L^{2}(0,1)}.
 \end{equation}
 In fact, for any $x\in(0,1)$
 \begin{equation*}
  f(x)-\int_{0}^{1}f(y)dy
  =
  \int_{0}^{1}(f(x)-f(y))dy
  =
  \int_{0}^{1}\int_{y}^{x}f_{x}(z)dzdy
 \end{equation*}
 hence we obtain by the H\"older inequality that
 \begin{equation*}
  \begin{split}
   \left|f(x)-\int_{0}^{1}f(y)dy\right|
   &\leq
   \left|\int_{0}^{1}\left|
   \int_{y}^{x}|f_{x}(z)|dz
   \right|dy\right|
   \\
   &\leq
   \left|\int_{0}^{1}\left|\int_{0}^{1}|f_{x}(z)|dz\right|dy\right|
   \leq
   \|f_{x}\|_{L^{2}(0,1)}.
  \end{split}
 \end{equation*}
 From \eqref{eq:LS.Sobolev_Embedding1}, for $u\in
 H^1_{\mathrm{per.ave}}(0,1)$,
 \begin{equation*}
  \label{eq:LS.Sobolev_Embedding_X}
  \sup_{x\in(0,1)}\left|u(x)\right|
  \leq
  \|u_{x}\|_{L^{2}(0,1)}
 \end{equation*}
 hence \eqref{eq:LS.def_InnerProduct} gives an inner product on $X$.

 Furthermore, for $u\in H^2_{\mathrm{per.ave}}(0,1)$, we can use
 \eqref{eq:LS.Sobolev_Embedding1} and
 \begin{equation*}
  \begin{split}
   \left|u_x(x)\right|
   &\leq
   \left|u_x(x)-\int_{0}^{1}u_x(y)dy\right|
   +
   \left|\int_{0}^{1}u_x(y)dy\right|
   \\
   &\leq
   \|u_{xx}\|_{L^{2}(0,1)}+\left|\int_{0}^{1}u_x(y)dy\right|
   \le\|u_{xx}\|_{L^{2}(0,1)}+\|u_x\|_{L^{2}(0,1)}
  \end{split}
 \end{equation*}
 for $x\in(0,1)$. Thus, for $u\in H^2_{\mathrm{per.ave}}(0,1)$, we
 obtain
 \begin{equation}
  \label{eq:LS.Sobolev_Embedding_Y}
   \sup_{x\in(0,1)}\left|u_x(x)\right|
   \leq
   \|u_{xx}\|_{L^{2}(0,1)}+\|u_x\|_{L^{2}(0,1)}
   \leq 
   \sqrt{2}
   \|u\|_{H_{\mathrm{per.ave}}^{2}(0,1)}.
 \end{equation}
\end{remark}

Now we are in a position to state the main theorem in this section.

\begin{theorem}
 \label{thm:Lojasiewicz-Simon}
 Let $\sigma$ be an analytic function on $\R$. Let $X$, $Y$ be Hilbert
 spaces defined by \eqref{eq:LS.Function_Spaces}. Assume
 \eqref{eq:1.Assumption_positivity}, namely there exists
 $\Cr{const:GBE_positivity}>0$ such that $\sigma(\alpha)\ge
 \Cr{const:GBE_positivity}$ for all $\alpha\in\R$. Let
 $(\overline{u},\overline{\alpha})\in Y$ be a critical point of
 $E$. Then, there exist a neighborhood $U$ of
 $(\overline{u},\overline{\alpha})$ in $Y$ and constants
 $0<\theta\leq\frac{1}{2}$, $\Cl{const:LS_GBE}>0$ such that
 \begin{equation}
  \label{eq:LS.LS_for_GBE}
  \|\dot{E}(u,\alpha)\|_{Y}
  \ge
  \Cr{const:LS_GBE}|E[u,\alpha]-E[\overline{u},\overline{\alpha}]|^{1-\theta}
 \end{equation}
 for all $(u,\alpha)\in U$.
 Here, the Fr\'{e}chet derivative $\dot{E}$ is explicitly written as
 \begin{multline*}
  \dot{E}(u,\alpha)
  =
  \biggl(
  \sigma(\alpha)
  \left\{
  \int_{0}^{x}\frac{u_{x}(y)}{\sqrt{1+u_{x}^{2}(y)}}dy-\int_{0}^{1}\int_{0}^{x}\frac{u_{x}(y)}{\sqrt{1+u_{x}^{2}(y)}}dydx
  \right\},
  \\
  \sigma'(\alpha)\int_{0}^{1}\sqrt{1+u_{x}^{2}(x)}\,dx
  \biggr).
 \end{multline*}
\end{theorem}

\subsection{Auxiliary lemma}

To prove Theorem \ref{thm:Lojasiewicz-Simon} by applying Theorem
\ref{thm:LS.LS}, we list an auxiliary lemma. First, we introduce the
sufficient condition for the Fr\'{e}chet differentiability.

\begin{lemma}
 [Proposition 1.3 in {\cite{MR4274456}}]
 \label{lem:LS.Gateau_Frechet}
 Let $X$ be a Banach space and let $O$ be a neighborhood of a vector
 $u\in X$. Consider an operator $F:O\rightarrow Y$ defined in $O$ into a
 Banach space $Y$. If $F$ is Gâteaux differentiable for any $v \in O$
 with derivative $DF(v,h)$, if $DF(v,h)$ is linear and continuous in $h$
 for any $v\in O$, and if $v\mapsto DF(v,\cdot)$ is continuous from $O$
 into $\mathscr{L}(X,Y)$, then $F$ is Fréchet differentiable at $u$.
\end{lemma}

Next, to compute the integrand of \eqref{eq:1.GBE}, we show several
estimates about $\sqrt{1+\xi^2}$.

\begin{lemma}
 \label{lem:LS.Vol_Aux_Inequality}
 For $\xi,\eta\in\R$,
 \begin{equation}
  \label{eq:LS.Vol_Aux_Inequality}
   \begin{aligned}
    \left|
    \sqrt{1+\xi^2}-\sqrt{1+\eta^2}
    \right|
    &\leq
    |\xi-\eta|,
    \\
    \left|
    \frac{\xi}{\sqrt{1+\xi^2}}
    -
    \frac{\eta}{\sqrt{1+\eta^2}}
    \right|
    &\leq
   |\xi-\eta|,
   \\
    \left|
    \frac{1}{(\sqrt{1+\xi^2})^3}
    -
    \frac{1}{(\sqrt{1+\eta^2})^3}
    \right|
    &\leq
    3|\xi-\eta|, \\
    \left|
    \frac{\xi}{(\sqrt{1+\xi^2})^5}
    -
    \frac{\eta}{(\sqrt{1+\eta^2})^5}
    \right|
    &\leq
    5|\xi-\eta|.
   \end{aligned}
 \end{equation}
\end{lemma}

\begin{proof}
 Let $f(\xi):=\sqrt{1+\xi^2}$ for $\xi\in\R$.  Taking a derivative of
 $f$, we have
 \begin{equation*}
   \begin{aligned}
     f'(\xi)&=\frac{\xi}{\sqrt{1+\xi^2}},
     &\quad
     f''(\xi)&=\frac{1}{(\sqrt{1+\xi^2})^3},
     \\
     f'''(\xi)&=-\frac{3\xi}{(\sqrt{1+\xi^2})^5},
     &\quad
     f''''(\xi)&=\frac{-3+12\xi^2}{(\sqrt{1+\xi^2})^7}.
   \end{aligned}
 \end{equation*}
 Then \eqref{eq:LS.Vol_Aux_Inequality} may be obtained by the
 fundamental calculus, $\sqrt{1+\xi^2}\geq 1$, and $\sqrt{1+\xi^2}\geq
 |\xi|$. For instance, since $|f''''(\tau)|\leq 15$ for $\tau\in\R$, we
 have
 \begin{equation*}
  \left|
   \frac{-3\xi}{(\sqrt{1+\xi^2})^5}
   -
   \frac{-3\eta}{(\sqrt{1+\eta^2})^5}
  \right|
  =
  |f'''(\xi)-f'''(\eta)|
  \leq
  \left|
   \int_\eta^\xi |f''''(\tau)|\,d\tau
  \right|
  \leq
  15|\xi-\eta|.
 \end{equation*}
\end{proof}

\subsection{Proof of the \L ojasiewicz-Simon inequality}

To prove Theorem \ref{thm:Lojasiewicz-Simon}, we first derive the
derivative of $E$.

\begin{proposition}
 \label{prop:LS.Conti_diff_X}
 Let $E$ be defined by \eqref{eq:1.GBE}. Then for $(u,\alpha)\in X$, the
 Fr\'{e}chet derivative of $E$ at $(u,\alpha)$ on $X$ is given by
 \begin{multline}
  \label{app:11}
  \dot{E}(u,\alpha)
  =
   \biggl(
   \sigma(\alpha)\left\{\int_{0}^{x}\frac{u_{x}(y)}{\sqrt{1+u_{x}^{2}(y)}}dy-\int_{0}^{1}\int_{0}^{x}\frac{u_{x}(y)}{\sqrt{1+u_{x}^{2}(y)}}dydx\right\},
  \\
   \sigma'(\alpha)\int_{0}^{1}\sqrt{1+u_{x}^{2}}\,dx
  \biggr).
 \end{multline}
\end{proposition}

\begin{proof}
 We will show that
 \begin{align*}
  E[u+h,\alpha+\beta]
  =
  E[u,\alpha]+(\dot{E}(u,\alpha),(h,\beta))_{X}+o(\|(h,\beta)\|_X)
 \end{align*}
 as $(h,\beta)\rightarrow0$ in $X$. First, we compute
 \begin{align}
  \label{app:6}
  \begin{split}
   E[u+h,\alpha+\beta]-E[u,\alpha]
   &=\int_{0}^{1}\{\sigma(\alpha+\beta)-\sigma(\alpha)\}\sqrt{1+(u+h)^{2}_{x}}\,dx\\
   &\qquad
   +
   \int_{0}^{1}
   \sigma(\alpha)\left\{\sqrt{1+(u+h)^{2}_{x}}
   -
   \sqrt{1+u_{x}^{2}}\right\}\,dx.
  \end{split}
 \end{align}

 Before computing the right-hand side of \eqref{app:6}, note by Taylor's
 theorem to the function $\sqrt{1+\xi^2}$ that for any $\xi,\eta\in\R$,
 there is $\theta_1\in(0,1)$ such that
 \begin{equation}
  \label{eq:3_Taylor_sqrt}
   \sqrt{1+(\xi+\eta)^2}
   =
   \sqrt{1+\xi^2}
   +
   \frac{\xi}{\sqrt{1+\xi^2}}\eta
   +
   \frac{1}{2(\sqrt{1+(\xi+\theta_1\eta)^2})^3}\eta^2.
 \end{equation}

 We consider the first term of the right-hand side of \eqref{app:6}.  We
 deduce by Taylor's theorem to $\sigma(\alpha+\beta)$ and $\xi=u_x$,
 $\eta=h_x$ to \eqref{eq:3_Taylor_sqrt} that
 \begin{equation*}
  \begin{split}
   &\qquad
   \int_{0}^{1}
   \{\sigma(\alpha+\beta)-\sigma(\alpha)\}\sqrt{1+(u+h)^{2}_{x}}\,dx
   \\
   &=
   \left(
   \sigma'(\alpha)\beta
   +
   \frac{1}{2}\sigma''(\alpha+\theta_2\beta)\beta^2
   \right)
   \\
   &\quad
   \times
   \left(
   \int_0^1
   \sqrt{1+u_x^2}\,dx
   +
   \int_0^1
   \frac{u_x}{\sqrt{1+u_x^2}}h_x
   \,dx
   +
   \int_0^1
   \frac{1}{2(\sqrt{1+(u_x+\theta_1h_x)^2})^3}h_x^2\,dx
   \right)
  \end{split}
 \end{equation*}
 for some $\theta_1,\ \theta_2\in(0,1)$. Since
 \begin{equation*}
  \left|
   \frac{u_x}{\sqrt{1+u_x^2}}
  \right|,\quad
  \left|
  \frac{1}{(\sqrt{1+(u_x+\theta_1h_x)^2})^3}
  \right|
  \leq1
 \end{equation*}
 for all $u,h\in H_{\text{per.ave}}^{1}(0,1)$, we find
 \begin{equation*}
   \left|
   \int_0^1
   \frac{u_x}{\sqrt{1+u_x^2}}h_x
   \,dx
   \right|
   \leq
   \|h_x\|_{L^2(0,1)}
   \leq
   \|(h,\beta)\|_X,
 \end{equation*}
 and
 \begin{equation}
  \label{eq:3.derivative_reminder}
   \left|
    \int_0^1
    \frac{1}{(\sqrt{1+(u_x+\theta_1h_x)^2})^3}h_x^2\,dx
   \right|
   \leq
   \|h_x\|^2_{L^2(0,1)}
   \leq
   \|(h,\beta)\|_X^2.
 \end{equation}
 Therefore, we figure out that
 \begin{equation}
  \label{eq:3_derivative_1}
  \begin{split}
   &\quad
   (\sigma(\alpha+\beta)-\sigma(\alpha))
   \int_0^1\sqrt{1+(u+h)^{2}_{x}}\,dx
   \\
   &=
   \left(
    \sigma'(\alpha)\int_0^1\sqrt{1+u_{x}^{2}}\,dx
   \right)
   \beta
   +o(\|(h,\beta)\|_X).
  \end{split}
 \end{equation}

 Next, we consider the second term of the right-hand side of
 \eqref{app:6}. Using \eqref{eq:3_Taylor_sqrt} with $\xi=u_x$ and
 $\eta=h_x$ again that
 \begin{equation*}
  \label{app:10}
   \begin{split}
    &\quad
    \int_{0}^{1}
    \sigma(\alpha)
    \left\{\sqrt{1+(u+h)^{2}_{x}}-\sqrt{1+u_{x}^{2}}
    \right\}
    \,dx
   \\
   &=
   \sigma(\alpha)
   \left(
   \int_{0}^{1}
   \frac{u_x}{\sqrt{1+u_x^2}}h_x\,dx
   +
   \int_0^1
   \frac{1}{(2\sqrt{1+(u_x+\theta_1h_x)^2})^3}h_x^2\,dx 
  \right).
   \end{split}
 \end{equation*}
 Using \eqref{eq:3.derivative_reminder}, we find
 \begin{equation}
  \label{eq:3_derivative_2}
   \begin{split}
    &\quad
    \int_{0}^{1}
    \sigma(\alpha)
    \left\{\sqrt{1+(u+h)^{2}_{x}(x)}-\sqrt{1+u_{x}^{2}(x)}
    \right\}
    dx
    \\
    &=
    \sigma(\alpha)
    \int_{0}^{1}
    \frac{u_x}{\sqrt{1+u_x^2}}h_x\,dx
    +
    o(\|(h,\beta)\|_X).
   \end{split} 
 \end{equation}
 Plugging \eqref{eq:3_derivative_1} and \eqref{eq:3_derivative_2} into
 \eqref{app:6}, we obtain
 \begin{equation}
  \label{app:22}
   \begin{split}
    &\quad
    E[u+h,\alpha+\beta]-E[u,\alpha]
    \\
    &=
    \left(
    \sigma'(\alpha)\int_0^1\sqrt{1+u_{x}^{2}}\,dx
    \right)
    \beta
    +
    \sigma(\alpha)
    \int_{0}^{1}
    \frac{u_x}{\sqrt{1+u_x^2}}h_x\,dx
    +
    o(\|(h,\beta)\|_X).
   \end{split} 
 \end{equation}

 From \eqref{app:22} and the inner product of
 \eqref{eq:LS.def_InnerProduct}, we will find for some constant $C$,
 \begin{equation*}
  \begin{split}
   \dot{E}(u,\alpha)
   =\left(
   \sigma(\alpha)
   \left(
   \int_0^x
   \frac{u_{x}(y)}{\sqrt{1+u_{x}^{2}(y)}}\,dy+C
   \right)
   ,
   \sigma'(\alpha)\int_{0}^{1}\sqrt{1+u_{x}^{2}}\,dx\right).
  \end{split}
 \end{equation*}
 To obtain the constant $C$, recall that the integral mean of the first
 component is $0$. Thus, we get
 \begin{equation*}
  C=-\int_{0}^{1}\int_{0}^{x}\frac{u_{x}(y)}{\sqrt{1+u_{x}^{2}(y)}}dydx.
 \end{equation*}
 Therefore, \eqref{app:11} is deduced.
\end{proof}

From Proposition~\ref{prop:LS.Conti_diff_X}, we obtain the necessary
condition for the critical point of $E$.

\begin{proposition}
 \label{prop:LS.Property_CriticalPt}
 Let $(\overline{u},\overline{\alpha})$ be a critical
 point of $E$, namely,
 $\dot{E}(\overline{u},\overline{\alpha})=0$. Then, we obtain
 \begin{align}
  \label{eq:Property_CP}
  \overline{u}=0,\quad
  \sigma'(\overline{\alpha})=0.
 \end{align}
\end{proposition}

\begin{proof}
 Since $(\overline{u},\overline{\alpha})\in X$ is a critical point of
 $E$, we obtain
 \begin{equation}
  \label{eq:3.CriticalPoint1}
   \sigma(\overline{\alpha})
    \left\{
     \int_{0}^{x}
     \frac{\overline{u}_{x}(y)}{\sqrt{1+\overline{u}_{x}^{2}(y)}}
     dy
     -
     \int_{0}^{1}
     \int_{0}^{x}
     \frac{\overline{u}_{x}(y)}{\sqrt{1+\overline{u}_{x}^{2}(y)}}
     dydx
    \right\}
    =0
 \end{equation}
 and
 \begin{equation}
  \label{eq:3.CriticalPoint2}
   \sigma'(\overline{\alpha})
   \int_{0}^{1}\sqrt{1+\overline{u}_{x}^{2}}\,dx
   =
   0.
 \end{equation}
 By the positivity assumption \eqref{eq:1.Assumption_positivity} of
 $\sigma$, 
 we find from the first component of
 \eqref{eq:3.CriticalPoint1} that,
 \begin{equation*}
  \int_{0}^{x}
  \frac{\overline{u}_{x}(y)}{\sqrt{1+\overline{u}_{x}^{2}(y)}}
  dy
  -
  \int_{0}^{1}
  \int_{0}^{x}\frac{\overline{u}_{x}(y)}{\sqrt{1+\overline{u}_{x}^{2}(y)}}
  dydx
  =
  0.
 \end{equation*}
 By taking a weak derivative by $x$, we get
 \begin{equation*}
  \frac{\overline{u}_{x}(x)}{\sqrt{1+\overline{u}_{x}^{2}(x)}}=0
  \quad
  \text{in}\ 
  L^2(0,1).
 \end{equation*}
 Therefore, we obtain $\overline{u}_{x}(x)=0$ for almost all
 $x\in(0,1)$, and $\overline{u}$ is a constant function. Since the
 integration of $\overline{u}$ is vanish, we have $\overline{u}=0$. From
 the the second component of \eqref{eq:3.CriticalPoint2}, we obtain
 $\sigma'(\overline{\alpha})=0$.
\end{proof}

Next, we compute the G\^{a}teaux derivative of $\dot{E}$, namely
$D[\dot{E}]((u,\alpha),(h,\beta))$ for $(h,\beta)\in X$.

\begin{proposition}
 \label{app:prop6}
 For $(u,\alpha),(h,\beta)\in X$, we obtain
 \begin{equation}
  \label{eq:3.SecondDerivativeGBE}
   \begin{aligned}
    &\quad
    D[\dot{E}]((u,\alpha),(h,\beta))
    \\
    &=
    \bigg(
    \int_{0}^{x}
    \biggl(
    \sigma'(\alpha)\beta
    \frac{u_{x}(y)}{\sqrt{1+u_{x}^{2}(y)}}
    +
    \sigma(\alpha)
    \frac{h_{x}(y)}{\left(\sqrt{1+u_{x}^{2}(y)}\right)^{3}}
    \biggr)dy
    \\
    &\quad
    -\int_{0}^{1}
    \int_{0}^{x}
    \biggl(
    \sigma'(\alpha)\beta
    \frac{u_{x}(y)}{\sqrt{1+u_{x}^{2}(y)}}
    +
    \sigma(\alpha)
    \frac{h_{x}(y)}{\left(\sqrt{1+u_{x}^{2}(y)}\right)^{3}}
    \biggr)dydx,
    \\
    &\quad
    \int_{0}^{1}
    \left(
    \sigma''(\alpha)\beta
    \sqrt{1+u_{x}^{2}(x)}
    +
    \sigma'(\alpha)
    \frac{u_{x}(x)h_{x}(x)}{\sqrt{1+u_{x}^{2}(x)}}
    \right)dx
    \bigg).
   \end{aligned} 
 \end{equation}
\end{proposition}

\begin{proof}
 We compute the G\^ateaux derivative of $\dot{E}$ at $(u,\alpha)\in
 X$. For $(u,\alpha),(h,\beta)\in X$,
 \begin{equation}
  \label{eq:3.second_derivative1}
   \begin{split}
    &\quad
    D[\dot{E}]((u,\alpha),(h,\beta))
    \\
    &=
    \frac{d}{d\theta}\bigg|_{\theta=0}
    \dot{E}(u+\theta h,\alpha+\theta\beta)
    \\
    &=
    \bigg(
    \frac{d}{d\theta}\bigg|_{\theta=0}
    \left(
    \sigma(\alpha+\theta\beta)
    \int_{0}^{x}\frac{(u+\theta h)_{x}(y)}{\sqrt{1+(u+\theta h)_{x}^{2}(y)}}dy\right) \\
    &\qquad
    -
    \frac{d}{d\theta}\bigg|_{\theta=0}
    \left(\sigma(\alpha+\theta\beta)\int_{0}^{1}\int_{0}^{x}\frac{(u+\theta h)_{x}(y)}{\sqrt{1+(u+\theta h)_{x}^{2}(y)}}dydx\right),
    \\
    &\qquad
    \frac{d}{d\theta}\bigg|_{\theta=0}
    \int_{0}^{1}
    \sigma'(\alpha+\theta\beta)
    \sqrt{1+(u+\theta h)_{x}^{2}}\,dx
    \bigg).
   \end{split}
 \end{equation}

 First, we compute 
 the first component of the right-hand side of
 \eqref{eq:3.second_derivative1}. Since
 \begin{multline*}
  \frac{d}{d\theta}
  \left(
  \sigma(\alpha+\theta\beta)
  \frac{(u+\theta h)_{x}(y)}{\sqrt{1+(u+\theta h)_{x}^{2}(y)}}
  \right)
  \\
  =
  \sigma'(\alpha+\theta\beta)
  \beta
  \frac{(u+\theta h)_{x}(y)}{\sqrt{1+(u+\theta h)_{x}^{2}(y)}}
  +
  \sigma(\alpha+\theta\beta)
  \frac{h_{x}(y)}{\left(\sqrt{1+(u+\theta h)_{x}^{2}(y)}\right)^{3}},
 \end{multline*}
 we have
 \begin{multline*}
  \frac{d}{d\theta}\bigg|_{\theta=0}
  \biggl(
  \sigma(\alpha+\theta\beta)
  \int_{0}^{x}\frac{(u+\theta h)_{x}(y)}{\sqrt{1+(u+\theta h)_{x}^{2}(y)}}dy
  -
  \\
 \sigma(\alpha+\theta\beta)
  \int_{0}^{1}\int_{0}^{x}\frac{(u+\theta h)_{x}(y)}{\sqrt{1+(u+\theta h)_{x}^{2}(y)}}dydx
  \biggr)
  \\
  =
  \int_{0}^{x}
  \left(
  \sigma'(\alpha)\beta\frac{u_{x}(y)}{\sqrt{1+u_{x}^{2}(y)}}
  +
  \sigma(\alpha)\frac{h_{x}(y)}{\left(\sqrt{1+u_{x}^{2}(y)}\right)^{3}}
  \right)
  dy
  \\
  -
  \int_{0}^{1}
  \int_{0}^{x}
  \left(
  \sigma'(\alpha)\beta\frac{u_{x}(y)}{\sqrt{1+u_{x}^{2}(y)}}
  +
  \sigma(\alpha)\frac{h_{x}(y)}{\left(\sqrt{1+u_{x}^{2}(y)}\right)^{3}}
  \right)
  dydx.
 \end{multline*}



 Next, we compute the second component of
 \eqref{eq:3.second_derivative1}. Direct computation tells us that,
 \begin{equation*}
  \begin{split}
   &\quad
   \frac{d}{d\theta}\bigg|_{\theta=0}
   \sigma'(\alpha+\theta\beta)
   \int_{0}^{1}\sqrt{1+(u+\theta h)_{x}^{2}}\,dx
   \\
   &=
   \int_{0}^{1}
   \left(
   \sigma''(\alpha)
   \beta
   \sqrt{1+u_{x}^{2}(x)}
   +
   \sigma'(\alpha)
   \frac{u_{x}(x)h_{x}(x)}{\sqrt{1+u_{x}^{2}(x)}}
   \right)
   dx.
  \end{split} 
 \end{equation*}
 Therefore \eqref{eq:3.SecondDerivativeGBE} is deduced.
\end{proof}

Now, we compute the G\^{a}teaux derivative of $\dot{E}$ at
$(\overline{u},\overline{\alpha})$.

\begin{proposition}
 \label{prop:LS.2nd_Derivative_at_CriticalPt}
 For any $(h,\beta)\in X$,
 \begin{align}
  \label{eq:CP_2nd_Derivative_GBE}
  D\dot{E}((\overline{u},\overline{\alpha}),(h,\beta))
  =
  (\sigma(\overline{\alpha})h,\sigma''(\overline{\alpha})\beta).
 \end{align}
\end{proposition}

\begin{proof}
 Plugging $(\overline{u}, \overline{\alpha})$ to
 \eqref{eq:3.SecondDerivativeGBE}, we obtain
 \begin{equation}
  \label{eq:CP_2nd_Derivative_GBE1}
   \begin{aligned}
    &\quad
    D\dot{E}((\overline{u},\overline{\alpha}),(h,\beta))
    \\
    &=
    \Biggl(
    \int_{0}^{x}
    \biggl\{
    \sigma'(\overline{\alpha})
    \beta
    \frac{\overline{u}_{x}(y)}{\sqrt{1+\overline{u}_{x}^{2}(y)}}
    +
    \sigma(\overline{\alpha})\frac{h_{x}(y)}{\left(\sqrt{1+\overline{u}_{x}^{2}(y)}
    \right)^{3}}
    \biggr\}dy
    \\
    &-
    \int_{0}^{1}
    \int_{0}^{x}
    \biggl\{
    \sigma'(\overline{\alpha})
    \beta
    \frac{\overline{u}_{x}(y)}{\sqrt{1+\overline{u}_{x}^{2}(y)}}
    +
    \sigma(\overline{\alpha})
    \frac{h_{x}(y)}{\left(\sqrt{1+\overline{u}_{x}^{2}(y)}
    \right)^{3}}
    \biggr\}
    dydx,\\
    &\int_{0}^{1}
    \biggl\{
    \sigma''(\overline{\alpha})
    \beta
    \sqrt{1+\overline{u}_{x}^{2}(x)}
    +
    \sigma'(\overline{\alpha})
    \frac{\overline{u}_{x}(x)h_{x}(x)}{\sqrt{1+\overline{u}_{x}^{2}(x)}}
    \biggr\}
    dx
    \Biggr).
   \end{aligned} 
 \end{equation}
 
 First, we compute the first term of the first component of
 \eqref{eq:CP_2nd_Derivative_GBE1}. Using \eqref{eq:Property_CP} and
 $h\in H_{\text{per.ave}}^{2}(0,1)$, we obtain
 \begin{align*}
  \int_{0}^{x}
  \biggl\{
  \sigma'(\overline{\alpha})
  \beta
  \frac{\overline{u}_{x}(y)}{\sqrt{1+\overline{u}_{x}^{2}(y)}}
  +
  \sigma(\overline{\alpha})
  \frac{h_{x}(y)}{\left(\sqrt{1+\overline{u}_{x}^{2}(y)}
  \right)^{3}}
  \biggr\}dy
  &=
  \int_{0}^{x}\sigma(\overline{\alpha})h_{x}(y)dy\\
  &=
  \sigma(\overline{\alpha})(h(x)-h(0)).
 \end{align*}

 Next, we compute the second term of the first component of
 \eqref{eq:CP_2nd_Derivative_GBE1}. Again using \eqref{eq:Property_CP}
 and $h\in H_{\text{per.ave}}^{2}(0,1)$, we obtain
  \begin{align*}
   &\quad
   \int_{0}^{1}
   \int_{0}^{x}
   \biggl\{
   \sigma'(\overline{\alpha})
   \beta
   \frac{\overline{u}_{x}(y)}{\sqrt{1+\overline{u}_{x}^{2}(y)}}
   +
   \sigma(\overline{\alpha})
   \frac{h_{x}(y)}{\left(\sqrt{1+\overline{u}_{x}^{2}(y)}\right)^{3}}
   \biggr\}
   dydx
   \\
   &=
   \int_{0}^{1}\sigma(\overline{\alpha})(h(x)-h(0))dx\\
   &=
   -\sigma(\overline{\alpha})h(0).
  \end{align*}

 Finally, we consider the second component of
 \eqref{eq:CP_2nd_Derivative_GBE1}. Using \eqref{eq:Property_CP} and
 $h\in H_{\text{per.ave}}^{2}(0,1)$, we obtain
 \begin{equation*}
  \int_{0}^{1}
   \biggl\{
   \sigma''(\overline{\alpha})\beta
   \sqrt{1+\overline{u}_{x}^{2}(x)}
   +
   \sigma'(\overline{\alpha})\frac{\overline{u}_{x}(x)h_{x}(x)}{\sqrt{1+\overline{u}_{x}^{2}(x)}}
   \biggr\}
   dx
   =
   \int_{0}^{1}\sigma''(\overline{\alpha})\beta dx
   =\sigma''(\overline{\alpha})\beta.
 \end{equation*}
 Combining the above computation, we get
 \eqref{eq:CP_2nd_Derivative_GBE}.
 \end{proof}

\subsubsection{Properties of the G\^ateaux derivative of $\dot{E}$}

To apply the abstract result for the \L ojasiewicz-Simon inequality, we
study the properties of the G\^ateaux derivative of $\dot{E}$. Hereafter
we denote
$L(h,\beta)=D\dot{E}((\overline{u},\overline{\alpha}),(h,\beta))$ for
$(h,\beta)\in X$. First, we show that $L$ is a bounded linear operator
on $X$.

\begin{proposition}
 $L$ is a bounded linear operator on $X$.
\end{proposition}

\begin{proof}
 It is easy to see that $L$ is a linear operator on $X$. We show the
 boundedness of $L$ on $X$.

 Since $\sigma(\overline{\alpha}), \sigma''(\overline{\alpha})$ is
 independent of $x$, we have
 \begin{equation*}
  \|\sigma(\overline{\alpha})h\|_{H_{\mathrm{per.ave}}^{1}(0,1)}\leq|\sigma(\overline{\alpha})|\|h\|_{H_{\mathrm{per.ave}}^{1}(0,1)},\quad
   |\sigma''(\overline{\alpha})\beta|=|\sigma''(\overline{\alpha})||\beta|.
 \end{equation*}
 Thus, for $(h,\beta)\in X$
 \begin{align*}
  \|L(h,\beta)\|_{X}^{2}
  &=\|\sigma(\overline{\alpha})h_x\|_{L^{2}(0,1)}^{2}
  +
  |\sigma''(\overline{\alpha})\beta|^{2}
  \\
  &\leq
  \max\{|\sigma(\overline{\alpha})|^{2},|\sigma''(\overline{\alpha})|^{2}\}
  \left(\|h\|_{H_{\text{per.ave}}^{1}(0,1)}^{2}+|\beta|^{2}\right)
  \\
  &=
  \max\{|\sigma(\overline{\alpha})|^{2},|\sigma''(\overline{\alpha})|^{2}\}
  \|(h,\beta)\|_{X}^{2}.
 \end{align*}
 Hence, $L$ is a bounded operator on $X$.
\end{proof}

Next, we show that $L$ is a Fredholm operator on $X$. 

 \begin{proposition}
  \label{prop:LS.L_Fredholm}
  $L$ is a Fredholm operator on $X$. Furthermore,
 \begin{equation}
  \label{eq:LS.Kernel_L}
  \Ker(L)
  =
  \begin{cases}
   \{0\}\times\R,&\text{if}\ \sigma''(\overline{\alpha})=0,\\
   \{0\}\times\{0\},&\text{if}\ \sigma''(\overline{\alpha})\neq 0
  \end{cases}
 \end{equation}
  and
  \begin{equation}
   \label{eq:LS.Image_L}
   \mathscr{R}(L)
   =
    \begin{cases}
     H_{\text{per.ave}}^{1}(0,1)\times\{0\},&
     \text{if}\ \sigma''(\overline{\alpha})=0,\\
     H_{\text{per.ave}}^{1}(0,1)\times\R,&
     \text{if}\ \sigma''(\overline{\alpha})\neq 0,
    \end{cases}
  \end{equation}
  where $\Ker L$ is the kernel of $L$ and $\mathscr{R}(L)$ is the range
  of $L$.
 \end{proposition}

\begin{proof}
 We need to show that
  \begin{align*}
   \mathscr{R}(L)\ \text{is closed on}\ X,\quad
   \dim(\Ker(L))<\infty,\quad
   \text{codim}(\mathscr{R}(L))<\infty.
  \end{align*}
 Since $\sigma(\overline{\alpha})\neq 0$, we have \eqref{eq:LS.Kernel_L}
 and \eqref{eq:LS.Image_L} from Proposition
 \ref{prop:LS.2nd_Derivative_at_CriticalPt}. Thus, $\mathscr{R}(L)$ is
 closed on $X$ and $\dim(\Ker(L))<\infty$. Further, since
  \begin{align*}
   \text{codim}(\mathscr{R}(L))
   =\dim(X/\mathscr{R}(L)),
  \end{align*}
 we have
  \begin{align*}
   X/\mathscr{R}(L)
   =
    \begin{cases}
     \{0\}\times\R,&
     \text{if}\ \sigma''(\overline{\alpha})=0, \\
     \{0\}\times\{0\},&
     \text{if}\ \sigma''(\overline{\alpha})\neq 0,
    \end{cases}
  \end{align*}
 hence
 $\text{codim}(\mathscr{R}(L))<\infty$. Therefore,
 $L$ is a Fredholm operator.
\end{proof}

\subsubsection{Properties of Banach space $Y$}

We next consider the properties of $Y$. Especially, we will show that
$Y$ satisfies the conditions \eqref{eq:LS.LS.Abstract_Assumption} in
Theorem \ref{thm:LS.LS}.

\begin{proposition}
 \label{prop:LS.Assumption_dense}
 $Y$ is dense and continuously embedded in $X$.
\end{proposition}

\begin{proof}
 It is clear that $Y$ is continuously embedded in $X$. We show that $Y$
 is dense in $X$. To prove this, it is enough to show that
 $H_{\mathrm{per.ave}}^{2}(0,1)$ is dense in
 $H_{\mathrm{per.ave}}^{1}(0,1)$. For any $u\in
 H_{\mathrm{per.ave}}^{1}(0,1)$, consider the Fourier expansion
 \begin{align*}
  u(x)
  =
  \frac{1}{2}a_{0}+\sum_{k=1}^{\infty}\left(a_{k}\cos\left(2\pi\left(x-\frac{1}{2}\right)\right)+b_{k}\sin\left(2\pi\left(x-\frac{1}{2}\right)\right)\right).
 \end{align*}
 Since
 \begin{align*}
  \int_{0}^{1}u(x)dx=0,
 \end{align*}
 we find $a_{0}=0$. Thus, for $N\in\N$, define
   \begin{align*}
   u_{N}(x)
   =
   \sum_{k=1}^{N}\left(a_{k}\cos\left(2\pi\left(x-\frac{1}{2}\right)\right)+b_{k}\sin\left(2\pi\left(x-\frac{1}{2}\right)\right)\right).
   \end{align*}
 It is clear that $u_{N}\in H_{\mathrm{per.ave}}^{2}(0,1)$. Since $u\in
 H_{\mathrm{per.ave}}^{1}(0,1)$,
 \begin{align*}
  u_{N}\to u\ \text{in}\ H_{\mathrm{per.ave}}^{1}(0,1),
 \end{align*}
 hence $H_{\mathrm{per.ave}}^{2}(0,1)$ is dense in
 $H_{\mathrm{per.ave}}^{1}(0,1)$.

\end{proof}

Next, we show that the critical point of $E$ is in $Y$.

\begin{proposition}
 \label{prop:LS.CriticalPt_in_Y}
 Let $(\overline{u},\overline{\alpha})\in X$ be a critical point of
 $E$.Then, $(\overline{u},\overline{\alpha})\in Y$.
\end{proposition}

\begin{proof}
 From Proposition \ref{prop:LS.Property_CriticalPt},
 $\overline{u}=0$. Therefore $(\overline{u},\overline{\alpha})\in Y$.
\end{proof}

Next, we show that the inverse image of $Y$ with respect to $L$ is a subset
of $Y$.

\begin{proposition}
 \label{prop:LS.inverse_image_L}
 $L^{-1}(Y)$ is a subset of $Y$.
\end{proposition}

\begin{proof}
 We will show that $(h,\beta)\in Y$ provided $L(h,\beta)\in Y$. Since
 $\beta\in\R$, it is enough to show $h\in
 H_{\mathrm{per.ave}}^{2}(0,1)$. By Proposition
 \ref{prop:LS.2nd_Derivative_at_CriticalPt}
 \begin{align*}
  L(h,\beta)
  =
  (
  \sigma(\overline{\alpha})h,\sigma''(\overline{\alpha})\beta),
 \end{align*}
 hence $\sigma(\overline{\alpha})h\in
 H_{\mathrm{per.ave}}^{2}(0,1)$. Because $\sigma(\overline{\alpha})\geq
 \Cr{const:GBE_positivity}\neq 0$ from
 \eqref{eq:1.Assumption_positivity}, we find $h\in
 H_{\mathrm{per.ave}}^{2}(0,1)$.
\end{proof}

Next, we study the image of $\dot{E}$ on $Y$.

\begin{proposition}
 \label{prop:LS.Image_on_Y}
 $\dot{E}(Y)$ is a subset of $Y$.
\end{proposition}

\begin{proof}
 By Proposition \ref{prop:LS.Conti_diff_X}, we will show for
 $(u,\alpha)\in Y$,
 \begin{equation}
  \label{eq:LS.Image_on_Y.proof1}
  \sigma(\alpha)
   \left\{
    \int_{0}^{x}\frac{u_{x}(y)}{\sqrt{1+u_{x}^{2}(y)}}dy-\int_{0}^{1}\int_{0}^{x}\frac{u_{x}(y)}{\sqrt{1+u_{x}^{2}(y)}}dydx
   \right\}
   \in H_{\mathrm{per.ave}}^{2}(0,1),
 \end{equation}
 and
 \begin{equation}
  \label{eq:LS.Image_on_Y.proof2}
  \sigma'(\alpha)\int_{0}^{1}\sqrt{1+u_{x}^{2}}\,dx\in\R.
 \end{equation}

 First, we show \eqref{eq:LS.Image_on_Y.proof1}.
 To do this, from
 \begin{equation*}
  \int_0^1
   \sigma(\alpha)
   \left\{\int_{0}^{\xi}
    \frac{u_{x}(y)}{\sqrt{1+u_{x}^{2}(y)}}
    dy
    -
    \int_{0}^{1}
    \int_{0}^{x}
    \frac{u_{x}(y)}{\sqrt{1+u_{x}^{2}(y)}}
    dydx\right\}
   \,d\xi=0,
 \end{equation*}
 the integral average of the first component of $\dot{E}(u,\alpha)$ has
 vanished. Next, we take a derivative with respect to $x$. Then, we
 obtain
  \begin{align*}
   \frac{d}{dx}\left(\sigma(\alpha)\int_{0}^{x}\frac{u_{x}(y)}{\sqrt{1+u_{x}^{2}(y)}}dy\right)
   =
   \sigma(\alpha)\frac{u_{x}(x)}{\sqrt{1+u_{x}^{2}(x)}}.
  \end{align*}
 Since $u\in H_{\mathrm{per.ave}}^{2}(0,1)$,
 $\sigma(\alpha)\frac{u_{x}(x)}{\sqrt{1+u_{x}^{2}(x)}}$ is a periodic
 function in $(0,1)$. Again using $u\in H_{\mathrm{per.ave}}^{2}(0,1)$
 and $\sqrt{1+u_x^2(x)}\geq 1$,
 \begin{align*}
  \int_{0}^{1}
  \left|\sigma(\alpha)\frac{u_{x}(x)}{\sqrt{1+u_{x}^{2}(x)}}\right|^{2}dx
  \leq|\sigma(\alpha)|^{2}\int_{0}^{1}|u_{x}(x)|^{2}dx
  <\infty.
 \end{align*}
  Thus,
  \begin{align*}
   \frac{d}{dx}\left(\sigma(\alpha)\int_{0}^{x}\frac{u_{x}(y)}{\sqrt{1+u_{x}^{2}(y)}}dy\right)
   \in L^{2}(0,1).
  \end{align*}
 Next, retaking a derivative, we find
 \begin{equation*}
  \frac{d^{2}}{dx^{2}}\left(\sigma(\alpha)\int_{0}^{x}\frac{u_{x}(y)}{\sqrt{1+u_{x}^{2}(y)}}dy\right)
   =
   \sigma(\alpha)\frac{u_{xx}(x)}{\left(\sqrt{1+u_{x}^{2}(x)}\right)^{3}}.
 \end{equation*}
 Since $u\in H_{\mathrm{per.ave}}^{2}(0,1)$ and
 $\sqrt{1+u_x^2(x)}\geq1$, we have
 \begin{align*}
  \int_{0}^{1}\left|\sigma(\alpha)\frac{u_{xx}(x)}{\left(\sqrt{1+u_{x}^{2}(x)}\right)^{3}}\right|^{2}dx
  \leq|\sigma(\alpha)|^{2}\int_{0}^{1}|u_{xx}(x)|^{2}dx
  <\infty,
 \end{align*}
 namely
 \begin{align*}
  \frac{d^{2}}{dx^{2}}\left(\sigma(\alpha)\int_{0}^{x}\frac{u_{x}(y)}{\sqrt{1+u_{x}^{2}(y)}}dy\right)
  \in L^{2}(0,1).
 \end{align*}
 Therefore \eqref{eq:LS.Image_on_Y.proof1} holds.
%
 From $u\in H_{\mathrm{per.ave}}^{2}(0,1)$, we find $u_{x}\in
 L^{2}(0,1)$, hence \eqref{eq:LS.Image_on_Y.proof2} holds.
 Therefore, we find $\dot{E}(Y)\subset Y$.
\end{proof}

Next, we show the fifth condition of
\eqref{eq:LS.LS.Abstract_Assumption}, namely the Fr\'{e}chet
differentiability of $\dot{E}$ on $Y$.
\begin{proposition}
 \label{prop:LS.Frechet_Y}
 $\dot{E}:Y\to Y$ is continuously Fr\'{e}chet differentiable on $Y$.
 \end{proposition}

To prove Proposition \ref{prop:LS.Frechet_Y}, we will apply Lemma
\ref{lem:LS.Gateau_Frechet}. To show Fr\'{e}chet differentiability
at $(\underline{u},\underline{\alpha})\in Y$, let
$O:=H^2_{\mathrm{per.ave}}(0,1)\times(\underline{\alpha}-1,\underline{\alpha}+1)\subset
Y$ be a neighborhood at $(\underline{u},\underline{\alpha})$ on $Y$. We
will show that $Y\ni(h,\beta)\mapsto
D\dot{E}((u,\alpha),(h,\beta))\in Y$ is a bounded linear operator for
any $(u,\alpha)\in O $ and $O\ni (u,\alpha)\mapsto
D\dot{E}((u,\alpha),\cdot)\in \mathscr{L}(Y)$ is continuous on $Y$.

\begin{lemma}
 \label{app:lem1} 
 For any fixed $(u,\alpha)\in O$, the G\^a{}teaux derivative
 $D\dot{E}((u,\alpha),\cdot):Y\to Y$ is a bounded linear operator on
 $Y$.
\end{lemma}

\begin{proof}
 It is clear that
 $D\dot{E}((u,\alpha),(h,\beta))$ is linear in
 $(h,\beta)$.
 Here, we will show that $D\dot{E}((u,\alpha),\cdot)$ is bounded
 operator on $Y$. For $(h,\beta)\in Y$, we denote 
 \begin{equation*}
  \|D\dot{E}((u,\alpha),(h,\beta))\|_Y^2
   =
   J_1^2+J_2^2,
 \end{equation*}
 where
 \begin{equation*}
  \begin{split}
   J_1
   &=
   \biggl\|
   \int_{0}^{x}\left\{\sigma'(\alpha)\beta\frac{u_{x}(y)}{\sqrt{1+u_{x}^{2}(y)}}+\sigma(\alpha)\frac{h_{x}(y)}{\left(\sqrt{1+u_{x}^{2}(y)}\right)^{3}}\right\}dy
   \\
   &\qquad
   -\int_{0}^{1}\int_{0}^{x}\left\{\sigma'(\alpha)\beta\frac{u_{x}(y)}{\sqrt{1+u_{x}^{2}(y)}}+\sigma(\alpha)\frac{h_{x}(y)}{\left(\sqrt{1+u_{x}^{2}(y)}\right)^{3}}\right\}dydx
   \biggr\|_{H^2_{\mathrm{per.ave}}(0,1)}, \\
   J_2
   &=
   \left|
   \int_{0}^{1}\left\{\sigma''(\alpha)\beta\sqrt{1+u_{x}^{2}(x)}+\sigma'(\alpha)\frac{u_{x}(x)h_{x}(x)}{\sqrt{1+u_{x}^{2}(x)}}\right\}dx\right|.
  \end{split}
 \end{equation*}

 First, we consider $J_1$. By taking the first and second derivative of the
 first component of $D\dot{E}((u,\alpha),(h,\beta))$, we obtain
 \begin{equation*}
  \begin{split}
   &\quad
   \frac{d}{dx}
   \left(
   \int_{0}^{x}
   \left\{
   \sigma'(\alpha)\beta\frac{u_{x}(y)}{\sqrt{1+u_{x}^{2}(y)}}
   +
   \sigma(\alpha)\frac{h_{x}(y)}{\left(\sqrt{1+u_{x}^{2}(y)}\right)^{3}}\right\}dy
   \right)
   \\
   &=
   \sigma'(\alpha)\beta\frac{u_{x}(x)}{\sqrt{1+u_{x}^{2}(x)}}
   +
   \sigma(\alpha)\frac{h_{x}(x)}{\left(\sqrt{1+u_{x}^{2}(x)}\right)^{3}},
  \end{split}
\end{equation*}
 and
 \begin{align*}
  &\quad
  \frac{d^{2}}{dx^{2}}
  \left(
  \int_{0}^{x}
  \left\{
  \sigma'(\alpha)\beta\frac{u_{x}(y)}{\sqrt{1+u_{x}^{2}(y)}}
  +
  \sigma(\alpha)\frac{h_{x}(x)}{\left(\sqrt{1+u_{x}^{2}(x)}\right)^{3}}\right\}dy
  \right)
  \\
  &=\sigma'(\alpha)\beta\frac{u_{xx}(x)}{\left(\sqrt{1+u_{x}^{2}(x)}\right)^{3}}
  +
  \sigma(\alpha)\left\{\frac{h_{xx}(x)}{\left(\sqrt{1+u_{x}^{2}(x)}\right)^{3}}
  -
  \frac{3h_{x}(x)u_{x}(x)u_{xx}(x)}{\left(\sqrt{1+u_{x}^{2}(x)}\right)^{5}}\right\}.
 \end{align*}
 We consider the first derivative of the first component of
 $D\dot{E}((u,\alpha),(h,\beta))$. Since $u\in
 H_{\text{per.ave}}^{2}(0,1)$ and $\sqrt{1+u_x^2}\geq1$,
 \begin{equation}
  \label{eq:LS.Geteau_bdd1}
   \left\|
    \sigma'(\alpha)\beta\frac{u_{x}(x)}{\sqrt{1+u_{x}^{2}(x)}}
   \right\|_{L^{2}(0,1)}
   \leq
   |\sigma'(\alpha)|\|u\|_{H^{2}_{\mathrm{per.ave}}(0,1)}|\beta|
 \end{equation}
 and
 \begin{equation}
  \label{eq:LS.Geteau_bdd2}
  \left\|
   \sigma(\alpha)\frac{h_{x}(x)}{\left(\sqrt{1+u_{x}^{2}(x)}\right)^{3}}
  \right\|_{L^{2}(0,1)}
  \leq
  |\sigma(\alpha)|\|h\|_{H^{2}_{\mathrm{per.ave}}(0,1)}.
 \end{equation}
 Next, we consider the second derivative of the first component of
 $D\dot{E}((u,\alpha),(h,\beta))$. Similarly, using $u\in
 H_{\text{per.ave}}^{2}(0,1)$, $\sqrt{1+u_x^2}\geq1$ and
 $\sqrt{1+u_x^2}\geq |u_x|$, we have
 \begin{equation}
  \label{eq:LS.Geteau_bdd3}
  \left\|
   \sigma'(\alpha)\beta\frac{u_{xx}}{(\sqrt{1+u_{x}^{2}})^3}
  \right\|_{L^{2}(0,1)}
  \leq
  |\sigma'(\alpha)|\|u\|_{H^{2}_{\mathrm{per.ave}}(0,1)}|\beta|
 \end{equation}
 and
 \begin{align*}
  &\quad
  \left\|
  \sigma(\alpha)
  \left\{
  \frac{h_{xx}(x)}{\left(\sqrt{1+u_{x}^{2}(x)}\right)^{3}}
  -
  \frac{3h_{x}(x)u_{x}(x)u_{xx}(x)}{\left(\sqrt{1+u_{x}^{2}(x)}\right)^{5}}
  \right\}
  \right\|_{L^{2}(0,1)}
  \\
  &\leq
  \left\|
  \sigma(\alpha)
  \frac{h_{xx}(x)}{\left(\sqrt{1+u_{x}^{2}(x)}\right)^{3}}
  \right\|_{L^{2}(0,1)}
  +
  3\left\|
  \sigma(\alpha)
  \frac{h_{x}(x)u_{x}(x)u_{xx}(x)}{\left(\sqrt{1+u_{x}^{2}(x)}
  \right)^{5}}
  \right\|_{L^{2}(0,1)}
  \\
  &\leq
  |\sigma(\alpha)|
  \left(
  \|h\|_{H^2_{\mathrm{per.ave}}(0,1)}
  +
  3
  \sup_{x\in(0,1)}|h_{x}(x)|
  \|u\|_{H^2_{\mathrm{per.ave}}(0,1)}
  \right).
 \end{align*}
 Using \eqref{eq:LS.Sobolev_Embedding_Y} with $h\in
 H^2_{\mathrm{per.ave}}(0,1)$, we have
 \begin{equation*}
  \sup_{x\in(0,1)}|h_{x}(x)|
  \leq
  \sqrt{2}
  \|h\|_{H^{2}_{\mathrm{per.ave}}(0,1)},
 \end{equation*}
 hence we obtain
 \begin{equation}
  \label{eq:LS.Geteau_bdd4}
   \begin{split}
    &\quad
    \left\|
    \sigma(\alpha)
    \left\{
    \frac{h_{xx}(x)}{\left(\sqrt{1+u_{x}^{2}(x)}\right)^{3}}
    -\frac{h_{x}(x)u_{x}(x)u_{xx}(x)}{\left(\sqrt{1+u_{x}^{2}(x)}\right)^{5}}\right\}\right\|_{L^{2}(0,1)}
    \\
    &\leq
    |\sigma(\alpha)|
    (
    1+
    3\sqrt{2}\|u\|_{H^2_{\mathrm{per.ave}}(0,1)}
    )
    \|h\|_{H^2_{\mathrm{per.ave}}(0,1)}.
   \end{split} 
 \end{equation}
 Therefore, we have from \eqref{eq:LS.Geteau_bdd1}, \eqref{eq:LS.Geteau_bdd2}, \eqref{eq:LS.Geteau_bdd3}, \eqref{eq:LS.Geteau_bdd4} that
 \begin{equation}
   \label{eq:LS.Geteau_bdd5}
    J_1
   \leq
   2|\sigma'(\alpha)|\|u\|_{H^{2}_{\mathrm{per.ave}}(0,1)}|\beta|
   +
   |\sigma(\alpha)|
   (
   2
   +
   3\sqrt{2}\|u\|_{H^2_{\mathrm{per.ave}}(0,1)}
   )
   \|h\|_{H^2_{\mathrm{per.ave}}(0,1)}.
 \end{equation}

 Next, we consider $J_2$. Using $u\in H_{\text{per.ave}}^{2}(0,1)$, the
 H\"older inequality, $\sqrt{1+u_x^2(x)}\geq 1$, and $1+u_x^2(x)\geq
 u^2_x(x)$, we have
  \begin{align*}
   \left|\int_{0}^{1}\sigma''(\alpha)\beta\sqrt{1+u_{x}^{2}}\,dx\right|
   &\leq|\sigma''(\alpha)||\beta|\left|\int_{0}^{1}
   \sqrt{1+u_{x}^{2}}\,dx\right|\\
   &\leq
   |\sigma''(\alpha)||\beta|\left(1+\|u\|_{H^{2}_{\mathrm{per.ave}}(0,1)}^{2}\right)^{\frac12},
  \end{align*}
 and
 \begin{align*}
  \left|
  \int_{0}^{1}
  \sigma'(\alpha)
  \frac{u_{x}(x)h_{x}(x)}{\sqrt{1+u_{x}^{2}(x)}}
  dx
  \right|
  &\leq
  |\sigma'(\alpha)|
  \left(
  \int_{0}^{1}\frac{u_{x}^{2}(x)}{1+u_{x}^{2}(x)}\,dx
  \right)^{\frac{1}{2}}
  \left(
  \int_{0}^{1}h_{x}^{2}(x)dx
  \right)^{\frac{1}{2}}
  \\
  &\leq
  |\sigma'(\alpha)\|h\|_{H_{\mathrm{per.ave}}^{2}(0,1)}.
 \end{align*}
 Hence
 \begin{equation}
  \label{eq:LS.Geteau_bdd6}
   J_2
   \leq
   |\sigma''(\alpha)||\beta|\left(1+\|u\|_{H^{2}_{\mathrm{per.ave}}(0,1)}^{2}\right)^{\frac12}
   +
   |\sigma'(\alpha)\|h\|_{H_{\mathrm{per.ave}}^{2}(0,1)}.
 \end{equation}

 Note that
 $|\sigma(\alpha)|,|\sigma'(\alpha)|,|\sigma''(\alpha)|,\|u\|_{H_{\text{per.ave}}^{2}(0,1)}<\infty$
 since $(u,\alpha)\in Y$. From \eqref{eq:LS.Geteau_bdd5},
 \eqref{eq:LS.Geteau_bdd6}, there is a constant $C>0$ such that
 \begin{align*}
  \|D\dot{E}((u,\alpha),(h,\beta))\|^2_Y
  =
  J_1^2+J_2^2
  &\leq C\|(h,\beta)\|_{Y}^2.
 \end{align*}
 Therefore, $Y\ni (h,\beta)\mapsto D\dot{E}((u,\alpha), (h,\beta))\in Y$ is a
 bounded operator on $Y$ for any $(u,\alpha)\in Y$.
\end{proof}

Next, we show the continuity of the derivative
$D\dot{E}((u,\alpha),\cdot)$ in terms of $(u,\alpha)$.

\begin{lemma}
 \label{app:lem2}
 $O\ni
 (u,\alpha)\mapsto D\dot{E}((u,\alpha),\cdot)\in \mathscr{L}(Y)$ is
 continuous on $Y$.
\end{lemma}

\begin{proof}
 In order to show that $(u,\alpha)\mapsto
 D\dot{E}((u,\alpha),\cdot)$ is
 continuous on $Y$, it is enough to prove
 \begin{align*}
  (u,\alpha)\to(\tilde{u},\tilde{\alpha})
  \ \text{in}\ Y
  \Rightarrow
  D\dot{E}((u,\alpha),\cdot)
  \rightarrow
  D\dot{E}((\tilde{u},\tilde{\alpha}),\cdot)
  \ \text{in}\;\mathscr{L}(Y).
 \end{align*}

 Since
 \begin{multline*}
  \sup_{\substack{(h,\beta)\in Y \\ (h,\beta)\neq 0}}
  \frac{
  \|
  D\dot{E}((u,\alpha),(h,\beta))
  -
  D\dot{E}((\tilde{u},\tilde{\alpha}),(h,\beta))
  \|_{Y}
  }
  {
  \|(h,\beta)\|_{Y}
  }
  \\
  =
  \|
  D\dot{E}((u,\alpha),\cdot)
  -
  D\dot{E}((\tilde{u},\tilde{\alpha}),\cdot)
  \|_{\mathscr{L}(Y)},
 \end{multline*}
 we will compute $\| D\dot{E}((u,\alpha),(h,\beta)) -
 D\dot{E}((\tilde{u},\tilde{\alpha}),(h,\beta)) \|_{Y}$ for all
 $(h,\beta)\in Y$ and derive the estimate of $ \|
 D\dot{E}((u,\alpha),\cdot) - D\dot{E}((\tilde{u},\tilde{\alpha}),\cdot)
 \|_{\mathscr{L}(Y)}$.

 From Proposition \ref{app:prop6}, 
 we write
 \begin{equation*}
  \|
  D\dot{E}((u,\alpha),(h,\beta))
 -
  D\dot{E}((\tilde{u},\tilde{\alpha}),(h,\beta))
  \|_{Y}^2
  =
  N^2_1+N^2_2,
 \end{equation*}
 where
 \begin{multline}
  \label{app:24}
  N_1
  :=
  \left
  \|\int_{0}^{x}\left\{\sigma'(\alpha)\beta\frac{u_{x}(y)}{\sqrt{1+u_{x}^{2}(y)}}-\sigma'(\tilde{\alpha})\beta\frac{\tilde{u}_{x}(y)}{\sqrt{1+\tilde{u}_{x}^{2}(y)}}\right\}dy\right.
  \\
  \qquad+\int_{0}^{x}\left\{\sigma(\alpha)\frac{h_{x}(y)}{\left(\sqrt{1+u_{x}^{2}(y)}\right)^{3}}-\sigma(\tilde{\alpha})\frac{h_{x}(y)}{\left(\sqrt{1+\tilde{u}_{x}^{2}(y)}\right)^{3}}\right\}dy\\
  -\int_{0}^{1}\int_{0}^{x}\left\{\sigma'(\alpha)\beta\frac{u_{x}(y)}{\sqrt{1+u_{x}^{2}(y)}}-\sigma'(\tilde{\alpha})\beta\frac{\tilde{u}_{x}(y)}{\sqrt{1+\tilde{u}_{x}^{2}(y)}}\right\}dydx
  \\
  -\left.\int_{0}^{1}\int_{0}^{x}\left\{\sigma(\alpha)\frac{h_{x}(y)}{\left(\sqrt{1+u_{x}^{2}(y)}\right)^{3}}-\sigma(\tilde{\alpha})\frac{h_{x}(y)}{\left(\sqrt{1+\tilde{u}_{x}^{2}(y)}\right)^{3}}\right\}dydx\right\|_{H_{\text{per.ave}}^{2}(0,1)}
 \end{multline}
 and
 \begin{equation}
  \begin{split}
   \label{app:25}
   N_2
   &:=
   \left|\int_{0}^{1}\left\{\sigma''(\alpha)\beta\sqrt{1+u_{x}^{2}(x)}-\sigma''(\tilde{\alpha})\beta\sqrt{1+\tilde{u}_{x}^{2}(x)}\right\}dx\right.\\
   &\qquad
   +\left.\int_{0}^{1}\left\{\sigma'(\alpha)\frac{u_{x}(x)h_{x}(x)}{\sqrt{1+u_{x}^{2}(x)}}-\sigma'(\tilde{\alpha})\frac{\tilde{u}_{x}(x)h_{x}(x)}{\sqrt{1+\tilde{u}_{x}^{2}(x)}}\right\}dx\right|.
  \end{split}
 \end{equation}
 Here, we compute \eqref{app:24} and \eqref{app:25} and derive the upper
  bounds of the operator norm $\| D\dot{E}((u,\alpha),\cdot) -
  D\dot{E}((\tilde{u},\tilde{\alpha}),\cdot) \|_{\mathscr{L}(Y)}$.

 First, we consider the first term of \eqref{app:24}. Taking the
 derivative twice, we have
 \begin{align*}
  &\quad\frac{d}{dx}\left(\int_{0}^{x}\left\{\sigma'(\alpha)\beta\frac{u_{x}(y)}{\sqrt{1+u_{x}^{2}(y)}}-\sigma'(\tilde{\alpha})\beta\frac{\tilde{u}_{x}(y)}{\sqrt{1+\tilde{u}_{x}^{2}(y)}}\right\}dy\right)\\
  &=
  \sigma'(\alpha)\beta\frac{u_{x}(x)}{\sqrt{1+u_{x}^{2}(x)}}-\sigma'(\tilde{\alpha})\beta\frac{\tilde{u}_{x}(x)}{\sqrt{1+\tilde{u}_{x}^{2}(x)}}
 \end{align*}
 and
 \begin{align*}
  &\quad\frac{d^{2}}{dx^{2}}\left(\int_{0}^{x}\left\{\sigma'(\alpha)\beta\frac{u_{x}(y)}{\sqrt{1+u_{x}^{2}(y)}}-\sigma'(\tilde{\alpha})\beta\frac{\tilde{u}_{x}(y)}{\sqrt{1+\tilde{u}_{x}^{2}(y)}}\right\}dy\right)\\
   &=\sigma'(\alpha)\beta\frac{u_{xx}(x)}{\left(\sqrt{1+u_{x}^{2}(x)}\right)^{3}}-\sigma'(\tilde{\alpha})\beta\frac{\tilde{u}_{xx}(x)}{\left(\sqrt{1+\tilde{u}_{x}^{2}(x)}\right)^{3}}.
 \end{align*}
 By the definition of the norm of $H_{\text{per.ave}}^{2}(0,1)$, we have
 \begin{equation}
  \label{eq:LS.Freschet_OperatorNorm1-1}
   \begin{split}
    &\quad\left\|\int_{0}^{x}\left\{\sigma'(\alpha)\beta\frac{u_{x}(y)}{\sqrt{1+u_{x}^{2}(y)}}-\sigma'(\tilde{\alpha})\beta\frac{\tilde{u}_{x}(y)}{\sqrt{1+\tilde{u}_{x}^{2}(y)}}\right\}dy\right\|_{H_{\text{per.ave}}^{2}(0,1)}
    \\
    &\leq|\beta|
    \left\|
    \sigma'(\alpha)\frac{u_{x}}{\sqrt{1+u_{x}^{2}}}
    -
    \sigma'(\tilde{\alpha})\frac{\tilde{u}_{x}}{\sqrt{1+\tilde{u}_{x}^{2}}}
    \right\|_{L^{2}(0,1)}
    \\
    &\quad
    +
    |\beta|
    \left\|
    \sigma'(\alpha)\frac{u_{xx}}{\left(\sqrt{1+u_{x}^{2}}\right)^{3}}
    -
    \sigma'(\tilde{\alpha})\frac{\tilde{u}_{xx}}{\left(\sqrt{1+\tilde{u}_{x}^{2}}\right)^{3}}\right\|_{L^{2}(0,1)}.
   \end{split}
 \end{equation}
 Next, we compute the second term of \eqref{app:24}. Taking the
 derivative twice, we obtain
 \begin{align*}
  &\quad\frac{d}{dx}\left(\int_{0}^{x}\left\{\sigma(\alpha)\frac{h_{x}(y)}{\left(\sqrt{1+u_{x}^{2}(y)}\right)^{3}}-\sigma(\tilde{\alpha})\frac{h_{x}(y)}{\left(\sqrt{1+\tilde{u}_{x}^{2}(y)}\right)^{3}}\right\}dy\right)\\
  &=\sigma(\alpha)\frac{h_{x}(x)}{\left(\sqrt{1+u_{x}^{2}(x)}\right)^{3}}-\sigma(\tilde{\alpha})\frac{h_{x}(x)}{\left(\sqrt{1+\tilde{u}_{x}^{2}(x)}\right)^{3}}
 \end{align*}
 and
\begin{align*}
 &\quad\frac{d^{2}}{dx^{2}}\left(\int_{0}^{x}\left\{\sigma(\alpha)\frac{h_{x}(y)}{\left(\sqrt{1+u_{x}^{2}(y)}\right)^{3}}-\sigma(\tilde{\alpha})\frac{h_{x}(y)}{\left(\sqrt{1+\tilde{u}_{x}^{2}(y)}\right)^{3}}\right\}dy\right)\\
 &=\left\{\sigma(\alpha)\frac{h_{xx}(x)}{\left(\sqrt{1+u_{x}^{2}(x)}\right)^{3}}-\sigma(\tilde{\alpha})\frac{h_{xx}(x)}{\left(\sqrt{1+\tilde{u}_{x}^{2}(x)}\right)^{3}}\right\}\\
 &\qquad-3\left\{\sigma(\alpha)\frac{h_{x}(x)u_{x}(x)u_{xx}(x)}{\left(\sqrt{1+u_{x}^{2}(x)}\right)^{5}}-\sigma(\tilde{\alpha})\frac{h_{x}(x)\tilde{u}_{x}(x)\tilde{u}_{xx}(x)}{\left(\sqrt{1+\tilde{u}_{x}^{2}(x)}\right)^{5}}\right\}.
\end{align*}

By the definition of the norm of $H^2_{\mathrm{per.ave}}(0,1)$, we have
 \begin{align*}
  &\quad\left\|\int_{0}^{x}\left\{\sigma(\alpha)\frac{h_{x}(y)}{\left(\sqrt{1+u_{x}^{2}(y)}\right)^{3}}-\sigma(\tilde{\alpha})\frac{h_{x}(y)}{\left(\sqrt{1+\tilde{u}_{x}^{2}(y)}\right)^{3}}\right\}dy\right\|_{H^{2}_{\mathrm{per.ave}}(0,1)}
  \\
  &\leq
  \left\|
  \sigma(\alpha)\frac{h_{x}}{\left(\sqrt{1+u_{x}^{2}}\right)^{3}}-\sigma(\tilde{\alpha})\frac{h_{x}}{\left(\sqrt{1+\tilde{u}_{x}^{2}}\right)^{3}}
  \right\|_{L^{2}(0,1)}\\
  &\quad
  +\left\|\sigma(\alpha)\frac{h_{xx}}{\left(\sqrt{1+u_{x}^{2}}\right)^{3}}-\sigma(\tilde{\alpha})\frac{h_{xx}}{\left(\sqrt{1+\tilde{u}_{x}^{2}}\right)^{3}}\right\|_{L^{2}(0,1)}\\
  &\quad
  +
  3
  \left\|\sigma(\alpha)\frac{h_{x}u_{x}u_{xx}}{\left(\sqrt{1+u_{x}^{2}}\right)^{5}}
  -
  \sigma(\tilde{\alpha})\frac{h_{x}\tilde{u}_{x}\tilde{u}_{xx}}{\left(\sqrt{1+\tilde{u}_{x}^{2}}\right)^{5}}\right\|_{L^{2}(0,1)}
  \\
  &\leq
  2\sup_{x\in(0,1)}
   \left|
  \sigma(\alpha)
  \frac{1}{\left(\sqrt{1+u_{x}^{2}(x)}\right)^{3}}
  -
  \sigma(\tilde{\alpha})\frac{1}{\left(\sqrt{1+\tilde{u}_{x}^{2}(x)}
  \right)^{3}}\right|
  \|h\|_{H^2_{\mathrm{per.ave}}(0,1)}
  \\
  &\quad
  +
  3\sup_{x\in(0,1)}|h_x(x)|
  \left\|\sigma(\alpha)\frac{u_{x}u_{xx}}{\left(\sqrt{1+u_{x}^{2}}\right)^{5}}-\sigma(\tilde{\alpha})\frac{\tilde{u}_{x}\tilde{u}_{xx}}{\left(\sqrt{1+\tilde{u}_{x}^{2}}\right)^{5}}\right\|_{L^{2}(0,1)}.
 \end{align*}
 Note from the Sobolev inequality \eqref{eq:LS.Sobolev_Embedding_Y},
 $\sup_{x\in(0,1)}|h_x(x)|\leq \sqrt{2}\|h\|_{H^2_{\mathrm{per.ave}}(0,1)}$,
 hence we find
 \begin{equation}
  \label{eq:LS.Freschet_OperatorNorm1-2}
   \begin{split}
    &\quad\left\|\int_{0}^{x}\left\{\sigma(\alpha)\frac{h_{x}(y)}{\left(\sqrt{1+u_{x}^{2}(y)}\right)^{3}}-\sigma(\tilde{\alpha})\frac{h_{x}(y)}{\left(\sqrt{1+\tilde{u}_{x}^{2}(y)}\right)^{3}}\right\}dy\right\|_{H^{2}_{\mathrm{per.ave}}(0,1)}
    \\
    &\leq
    \Biggl(
    2\sup_{x\in(0,1)}
    \left|
    \sigma(\alpha)\frac{1}{\left(\sqrt{1+u_{x}^{2}(x)}\right)^{3}}-\sigma(\tilde{\alpha})\frac{1}{\left(\sqrt{1+\tilde{u}_{x}^{2}(x)}\right)^{3}}
    \right|
    \\
    &\quad
    +
    3\sqrt{2}
    \left\|
    \sigma(\alpha)\frac{u_{x}u_{xx}}{\left(\sqrt{1+u_{x}^{2}}\right)^{5}}-\sigma(\tilde{\alpha})\frac{\tilde{u}_{x}\tilde{u}_{xx}}{\left(\sqrt{1+\tilde{u}_{x}^{2}}\right)^{5}}
    \right\|_{L^{2}(0,1)}
    \Biggr)
    \|h\|_{H^2_{\mathrm{per.ave}}(0,1)}.
   \end{split}
 \end{equation}
 Since the third and fourth terms in the norm of \eqref{app:24} are
 constants, we obtain from \eqref{eq:LS.Freschet_OperatorNorm1-1} and
 \eqref{eq:LS.Freschet_OperatorNorm1-2} that
 \begin{equation}
  \label{eq:LS.Freschet_OperatorNorm1}
   \begin{split}
   \frac{N_1}{\|(h,\beta)\|_Y}
   &\leq
   \left\|
   \sigma'(\alpha)\frac{u_{x}}{\sqrt{1+u_{x}^{2}}}
   -
   \sigma'(\tilde{\alpha})\frac{\tilde{u}_{x}}{\sqrt{1+\tilde{u}_{x}^{2}}}
   \right\|_{L^{2}(0,1)}
   \\
   &\quad
   +
   \left\|
   \sigma'(\alpha)\frac{u_{xx}}{\left(\sqrt{1+u_{x}^{2}}\right)^{3}}
   -
   \sigma'(\tilde{\alpha})\frac{\tilde{u}_{xx}}{\left(\sqrt{1+\tilde{u}_{x}^{2}}\right)^{3}}\right\|_{L^{2}(0,1)}
   \\
   &\quad
   +
   2\sup_{x\in(0,1)}
   \left|\sigma(\alpha)\frac{1}{\left(\sqrt{1+u_{x}^{2}(x)}\right)^{3}}-\sigma(\tilde{\alpha})\frac{1}{\left(\sqrt{1+\tilde{u}_{x}^{2}(x)}\right)^{3}}\right|
   \\
   &\quad
   +3\sqrt{2}
   \left\|\sigma(\alpha)\frac{u_{x}u_{xx}}{\left(\sqrt{1+u_{x}^{2}}\right)^{5}}-\sigma(\tilde{\alpha})\frac{\tilde{u}_{x}\tilde{u}_{xx}}{\left(\sqrt{1+\tilde{u}_{x}^{2}}\right)^{5}}\right\|_{L^{2}(0,1)}.
   \end{split}
 \end{equation}

 Next, we consider \eqref{app:25}. The first term can be transformed
  into
 \begin{align*}
  &\quad\left|\int_{0}^{1}\left\{\sigma''(\alpha)\beta\sqrt{1+u_{x}^{2}(x)}-\sigma''(\tilde{\alpha})\beta\sqrt{1+\tilde{u}_{x}^{2}(x)}\right\}dx\right|\\
  &\leq
  |\beta|
  \left|\int_{0}^{1}\left\{\sigma''(\alpha)\sqrt{1+u_{x}^{2}(x)}-\sigma''(\tilde{\alpha})\sqrt{1+\tilde{u}_{x}^{2}(x)}\right\}dx\right|.
 \end{align*}
 The second term can be computed together with
\eqref{eq:LS.Sobolev_Embedding_Y} as
 \begin{align*}
  &\quad\left|\int_{0}^{1}\left\{\sigma'(\alpha)\frac{u_{x}(x)h_{x}(x)}{\sqrt{1+u_{x}^{2}(x)}}-\sigma'(\tilde{\alpha})\frac{\tilde{u}_{x}(x)h_{x}(x)}{\sqrt{1+\tilde{u}_{x}^{2}(x)}}\right\}dx\right|\\
  &\leq\sup_{x\in(0,1)}|h_{x}(x)|\left|\int_{0}^{1}\left\{\sigma'(\alpha)\frac{u_{x}(x)}{\sqrt{1+u_{x}^{2}(x)}}-\sigma'(\tilde{\alpha})\frac{\tilde{u}_{x}(x)}{\sqrt{1+\tilde{u}_{x}^{2}(x)}}\right\}dx\right|\\
   &\leq\|h\|_{H^{2}_{\mathrm{per.ave}}(0,1)}
  \left|\int_{0}^{1}\left\{\sigma'(\alpha)\frac{u_{x}(x)}{\sqrt{1+u_{x}^{2}(x)}}-\sigma'(\tilde{\alpha})\frac{\tilde{u}_{x}(x)}{\sqrt{1+\tilde{u}_{x}^{2}(x)}}\right\}dx\right|.
 \end{align*}
 Therefore, $N_2$ can be estimated as
 \begin{equation}
  \label{eq:LS.Freschet_OperatorNorm2}
  \begin{split}
   \frac{N_2}{\|(h,\beta)\|_Y}
   &\leq
   \left|\int_{0}^{1}\left\{\sigma''(\alpha)\sqrt{1+u_{x}^{2}(x)}-\sigma''(\tilde{\alpha})\sqrt{1+\tilde{u}_{x}^{2}(x)}\right\}dx\right|
   \\
   &\quad
   +
   \left|\int_{0}^{1}\left\{\sigma'(\alpha)\frac{u_{x}(x)}{\sqrt{1+u_{x}^{2}(x)}}-\sigma'(\tilde{\alpha})\frac{\tilde{u}_{x}(x)}{\sqrt{1+\tilde{u}_{x}^{2}(x)}}\right\}dx\right|.
  \end{split}
 \end{equation}
 Combining \eqref{eq:LS.Freschet_OperatorNorm1} and
 \eqref{eq:LS.Freschet_OperatorNorm2}, we arrive at
 \begin{equation}
  \label{eq:LS.Freschet_OperatorNorm}
  \begin{split}
   &\quad
   \|
   D\dot{E}((u,\alpha),\cdot)
   -
   D\dot{E}((\tilde{u},\tilde{\alpha}),\cdot)
   \|_{\mathscr{L}(Y)}
   \\
   &\leq
   \left\|
   \sigma'(\alpha)\frac{u_{x}}{\sqrt{1+u_{x}^{2}}}
   -
   \sigma'(\tilde{\alpha})\frac{\tilde{u}_{x}}{\sqrt{1+\tilde{u}_{x}^{2}}}
   \right\|_{L^{2}(0,1)}
   \\
   &\quad
   +
   \left\|
   \sigma'(\alpha)\frac{u_{xx}}{\left(\sqrt{1+u_{x}^{2}}\right)^{3}}
   -
   \sigma'(\tilde{\alpha})\frac{\tilde{u}_{xx}}{\left(\sqrt{1+\tilde{u}_{x}^{2}}\right)^{3}}\right\|_{L^{2}(0,1)}
   \\
   &\quad
   +
   2\sup_{x\in(0,1)}
   \left|\sigma(\alpha)\frac{1}{\left(\sqrt{1+u_{x}^{2}(x)}\right)^{3}}-\sigma(\tilde{\alpha})\frac{1}{\left(\sqrt{1+\tilde{u}_{x}^{2}(x)}\right)^{3}}\right|
   \\
   &\quad
   +3\sqrt{2}
   \left\|\sigma(\alpha)\frac{u_{x}u_{xx}}{\left(\sqrt{1+u_{x}^{2}}\right)^{5}}-\sigma(\tilde{\alpha})\frac{\tilde{u}_{x}\tilde{u}_{xx}}{\left(\sqrt{1+\tilde{u}_{x}^{2}}\right)^{5}}\right\|_{L^{2}(0,1)}
   \\
   &\quad
   +
   \left|\int_{0}^{1}\left\{\sigma''(\alpha)\sqrt{1+u_{x}^{2}(x)}-\sigma''(\tilde{\alpha})\sqrt{1+\tilde{u}_{x}^{2}(x)}\right\}dx\right|
   \\
   &\quad
   +
   \left|\int_{0}^{1}\left\{\sigma'(\alpha)\frac{u_{x}(x)}{\sqrt{1+u_{x}^{2}(x)}}-\sigma'(\tilde{\alpha})\frac{\tilde{u}_{x}(x)}{\sqrt{1+\tilde{u}_{x}^{2}(x)}}\right\}dx\right|.
  \end{split}
 \end{equation}
 Therefore, we will show that the right-hand side of
 \eqref{eq:LS.Freschet_OperatorNorm} converges to $0$ as
 $(u,\alpha)\rightarrow(\tilde{u},\tilde{\alpha})$.

 To show the convergence, we first consider the first term of
 \eqref{eq:LS.Freschet_OperatorNorm}. To prove this, we
 note that
 \begin{equation*}
  \begin{split}
   &\qquad
   \sigma'(\alpha)\frac{u_{x}(x)}{\sqrt{1+u_{x}^{2}(x)}}
   -
   \sigma'(\tilde{\alpha})\frac{\tilde{u}_{x}(x)}{\sqrt{1+\tilde{u}_{x}^{2}(x)}}
   \\
   &=
   (\sigma'(\alpha)-\sigma'(\tilde{\alpha}))\frac{u_{x}(x)}{\sqrt{1+u_{x}^{2}(x)}}
   +
   \sigma'(\tilde{\alpha})
   \left(
   \frac{u_{x}(x)}{\sqrt{1+u_{x}^{2}(x)}}
   -
   \frac{\tilde{u}_{x}(x)}{\sqrt{1+\tilde{u}_{x}^{2}(x)}}
   \right).
  \end{split}
 \end{equation*}
 Let $M_{2}=\sup_{|t-\underline{\alpha}|\leq 1}|\sigma''(t)|$. Using
 $\sqrt{1+u_x^2}\geq |u_x|$ and $M_{2}<\infty$, we obtain
 \begin{align*}
  \left\|(\sigma'(\alpha)-\sigma'(\tilde{\alpha}))\frac{u_{x}}{\sqrt{1+u_{x}^{2}}}\right\|_{L^{2}(0,1)}
  &\leq\left|\int_{\tilde{\alpha}}^{\alpha}|\sigma''(t)|dt\right|
  \\
  &\leq M_{2}|\alpha-\tilde{\alpha}|
  \rightarrow 0
  \quad
  \text{as}\ 
  (u,\alpha)\to(\tilde{u},\tilde{\alpha}).
 \end{align*}
 Next using Lemma \ref{lem:LS.Vol_Aux_Inequality},
 \begin{equation*}
  \begin{split}
   \left\|
   \sigma'(\tilde{\alpha})
   \left(
   \frac{u_{x}}{\sqrt{1+u_{x}^{2}}}
   -
   \frac{\tilde{u}_{x}}{\sqrt{1+\tilde{u}_{x}^{2}}}
   \right)
   \right\|_{L^2(0,1)}
   &\leq
   |\sigma'(\tilde{\alpha})|
   \|u_{x}-\tilde{u}_{x}\|_{L^2(0,1)}
   \\
   &\leq
   |\sigma'(\tilde{\alpha})|
   \|u-\tilde{u}\|_{H^2_{\mathrm{per.ave}}(0,1)}
   \rightarrow0
  \end{split}
 \end{equation*}
 as $(u,\alpha)\rightarrow(\tilde{u},\tilde{\alpha})$. Hence we have
 \begin{equation}
  \label{eq:LS.conti-Y-1}
   \left\|
    \sigma'(\alpha)\frac{u_{x}}{\sqrt{1+u_{x}^{2}}}
    -
    \sigma'(\tilde{\alpha})\frac{\tilde{u}_{x}}{\sqrt{1+\tilde{u}_{x}^{2}}}
    \right\|_{L^{2}(0,1)}
    \rightarrow0
    \quad
    \text{as}\ 
    (u,\alpha)\rightarrow(\tilde{u},\tilde{\alpha}).
 \end{equation}

 Next, we consider the second term of the right-hand side of
 \eqref{eq:LS.Freschet_OperatorNorm}. We transform the integrand as
 \begin{align*}
  &\quad
  \sigma'(\alpha)\frac{u_{xx}(x)}{\left(\sqrt{1+u_{x}^{2}(x)}\right)^{3}}-\sigma'(\tilde{\alpha})\frac{\tilde{u}_{xx}(x)}{\left(\sqrt{1+\tilde{u}_{x}^{2}(x)}\right)^{3}}
  \\
   &=(\sigma'(\alpha)-\sigma'(\tilde{\alpha}))\frac{u_{xx}(x)}{\left(\sqrt{1+u_{x}^{2}(x)}\right)^{3}}+\sigma'(\tilde{\alpha})\left(\frac{u_{xx}(x)}{\left(\sqrt{1+u_{x}^{2}(x)}\right)^{3}}-\frac{\tilde{u}_{xx}(x)}{\left(\sqrt{1+\tilde{u}_{x}^{2}(x)}\right)^{3}}\right)
  \\
  &=(\sigma'(\alpha)-\sigma'(\tilde{\alpha}))\frac{u_{xx}(x)}{\left(\sqrt{1+u_{x}^{2}(x)}\right)^{3}}
  \\
  &\qquad+\sigma'(\tilde{\alpha})\frac{u_{xx}(x)-\tilde{u}_{xx}(x)}{\left(\sqrt{1+u_{x}^{2}(x)}\right)^{3}}
  \\
  &\qquad
  +
  \sigma'(\tilde{\alpha})\tilde{u}_{xx}(x)\left(\frac{1}{\left(\sqrt{1+u_{x}^{2}(x)}\right)^{3}}-\frac{1}{\left(\sqrt{1+\tilde{u}_{x}^{2}(x)}\right)^{3}}\right).
 \end{align*}
 Note that $\sqrt{1+u_x^2}\geq1$, we have
 \begin{align*}
  \left\|
  (\sigma'(\alpha)-\sigma'(\tilde{\alpha}))
  \frac{u_{xx}}{\left(\sqrt{1+u_{x}^{2}}\right)^{3}}
  \right\|_{L^{2}(0,1)}
  &\leq
  \left|
  \int_{\tilde{\alpha}}^{\alpha}|\sigma''(t)|dt
  \right|\|u_{xx}\|_{L^{2}(0,1)}
  \\
  &\leq
  M_{2}
  |\alpha-\tilde{\alpha}|\|u\|_{H^{2}_{\mathrm{per.ave}}(0,1)}
  \rightarrow 
  0\quad
 \end{align*}
 as $(u,\alpha)\to(\tilde{u},\tilde{\alpha})$. 
 Next, again noting that
 $\sqrt{1+u_{x}^{2}}\ge 1$, we obtain
 \begin{align*}
  \left\|
  \sigma'(\tilde{\alpha})
  \frac{u_{xx}-\tilde{u}_{xx}}{\left(\sqrt{1+u_{x}^{2}}\right)^{3}}
  \right\|_{L^2(0,1)}
  &\leq
  |\sigma'(\tilde{\alpha})|\|u_{xx}-\tilde{u}_{xx}\|_{L^{2}(0,1)}\\
  &\leq
  |\sigma'(\tilde{\alpha})|\|u-\tilde{u}\|_{H^{2}_{\mathrm{per.ave}}(0,1)}
  \rightarrow 0
  \quad
  \text{as}
  \ (u,\alpha)\to(\tilde{u},\tilde{\alpha}).
 \end{align*}
 Next, using Lemma~\ref{lem:LS.Vol_Aux_Inequality}, and the Sobolev
 inequality \eqref{eq:LS.Sobolev_Embedding_Y}, we have
 for any $x\in(0,1)$,
 \begin{equation*}
  \left|\frac{1}{\left(\sqrt{1+u_{x}^{2}(x)}\right)^{3}}-\frac{1}{\left(\sqrt{1+\tilde{u}_{x}^{2}(x)}\right)^{3}}\right|
  \leq 3|u_{x}(x)-\tilde{u}_{x}(x)|
  \leq
  3\sqrt{2}
  \|u-\tilde{u}\|_{H^{2}_{\mathrm{per.ave}}(0,1)}.
 \end{equation*}
 Thus, we obtain
 \begin{align*}
  &\quad
  \left\|
  \sigma'(\tilde{\alpha})\tilde{u}_{xx}\left(\frac{1}{\left(\sqrt{1+u_{x}^{2}}\right)^{3}}-\frac{1}{\left(\sqrt{1+\tilde{u}_{x}^{2}}\right)^{3}}\right)
  \right\|_{L^{2}(0,1)}\\
  &\leq
  3\sqrt{2}
  |\sigma'(\tilde{\alpha})|
  \|u\|_{H^{2}_{\mathrm{per.ave}}(0,1)}\|u-\tilde{u}\|_{H^{2}_{\mathrm{per.ave}}(0,1)}
  \rightarrow 0
  \quad
  \text{as}
  \ (u,\alpha)\to(\tilde{u},\tilde{\alpha}).
 \end{align*}
 Combining the above estimates, we arrive at
 \begin{equation}
  \label{eq:LS.conti-Y-2}
  \left\|
   \sigma'(\alpha)\frac{u_{xx}}{\left(\sqrt{1+u_{x}^{2}}\right)^{3}}
   -
   \sigma'(\tilde{\alpha})\frac{\tilde{u}_{xx}}{\left(\sqrt{1+\tilde{u}_{x}^{2}}\right)^{3}}\right\|_{L^{2}(0,1)}
  \rightarrow 0
 \end{equation}
 as $(u,\alpha)\to(\tilde{u},\tilde{\alpha})$.

 Next, we compute the third term of the right-hand side of
 \eqref{eq:LS.Freschet_OperatorNorm}.  To show the convergence, we first
 observe that
  \begin{align*}
   &\quad\sigma(\alpha)\frac{1}{\left(\sqrt{1+u_{x}^{2}(x)}\right)^{3}}-\sigma(\tilde{\alpha})\frac{1}{\left(\sqrt{1+\tilde{u}_{x}^{2}(x)}\right)^{3}}\\
   &=(\sigma(\alpha)-\sigma(\tilde{\alpha}))\frac{1}{\left(\sqrt{1+u_{x}^{2}(x)}\right)^{3}}+\sigma(\tilde{\alpha})\left(\frac{1}{\left(\sqrt{1+u_{x}^{2}(x)}\right)^{3}}-\frac{1}{\left(\sqrt{1+\tilde{u}_{x}^{2}(x)}\right)^{3}}\right).
  \end{align*}
 Let $M_{1}=\sup_{|t-\underline{\alpha}|<1}|\sigma'(t)|$. Since
 $\sqrt{1+u_x^2(x)}\geq1$, we obtain
 \begin{align*}
   \left|(\sigma(\alpha)-\sigma(\tilde{\alpha}))\frac{1}{\left(\sqrt{1+u_{x}^{2}(x)}\right)^{3}}\right|
   \leq\left|\int_{\tilde{\alpha}}^{\alpha}|\sigma'(t)|dt\right|
  \leq 
  M_{1} |\alpha-\tilde{\alpha}|.
 \end{align*}
 Next, using Lemma \ref{lem:LS.Vol_Aux_Inequality} together with the
 Sobolev inequality \eqref{eq:LS.Sobolev_Embedding_Y}, we have
 for $x\in(0,1)$
 \begin{equation*}
  \begin{split}
   \left|
   \sigma(\tilde{\alpha})\left(\frac{1}{\left(\sqrt{1+u_{x}^{2}(x)}\right)^{3}}-\frac{1}{\left(\sqrt{1+\tilde{u}_{x}^{2}(x)}\right)^{3}}\right)
   \right|
   &\leq
  3|\sigma(\tilde{\alpha})||u_{x}(x)-\tilde{u}_{x}(x)|
   \\
   &\leq
  3\sqrt{2}
  |\sigma(\tilde{\alpha})|\|u-\tilde{u}\|_{H^2_{\mathrm{per.ave}}(0,1)}.
  \end{split} 
\end{equation*}
  Combining these estimates, we obtain
  \begin{multline}
   \label{eq:LS.conti-Y-3}
    \sup_{x\in(0,1)}
   \left|\sigma(\alpha)\frac{1}{\left(\sqrt{1+u_{x}^{2}(x)}\right)^{3}}-\sigma(\tilde{\alpha})\frac{1}{\left(\sqrt{1+\tilde{u}_{x}^{2}(x)}\right)^{3}}\right|
   \\
   \leq
   M_{1} |\alpha-\tilde{\alpha}|
   +
   3\sqrt{2}
     |\sigma(\tilde{\alpha})|\|u-\tilde{u}\|_{H^2_{\mathrm{per.ave}}(0,1)}
     \to 0
\end{multline}
  as $(u,\alpha)\to(\tilde{u},\tilde{\alpha})$.

 Next, we compute the fourth term of the right-hand side of
 \eqref{eq:LS.Freschet_OperatorNorm}. Observe that
 \begin{align*}
  &\quad
  \sigma(\alpha)
  \frac{u_{x}(x)u_{xx}(x)}{\left(\sqrt{1+u_{x}^{2}(x)}\right)^{5}}
  -
  \sigma(\tilde{\alpha})
  \frac{\tilde{u}_{x}(x)\tilde{u}_{xx}(x)}{\left(\sqrt{1+\tilde{u}_{x}^{2}(x)}\right)^{5}}
  \\
  &=
  (\sigma(\alpha)-\sigma(\tilde{\alpha}))\frac{u_{x}(x)u_{xx}(x)}{\left(\sqrt{1+u_{x}^{2}(x)}\right)^{5}}\\
  &\qquad
  +
  \sigma(\tilde{\alpha})
  (u_{xx}(x)-\tilde{u}_{xx}(x))
  \left(
  \frac{u_{x}(x)}{\left(\sqrt{1+u_{x}^{2}(x)}\right)^{5}}
  \right)
  \\
  &\qquad
  +
  \sigma(\tilde{\alpha})
  \tilde{u}_{xx}(x)
  \left(
  \frac{u_x(x)}{\left(\sqrt{1+u_{x}^{2}(x)}\right)^{5}}
  -
  \frac{\tilde{u}_{x}(x)}{\left(\sqrt{1+\tilde{u}_{x}^{2}(x)}\right)^{5}}
  \right).
 \end{align*}
 Since $\sqrt{1+u_x^2(x)}\geq \max\{|u_{x}(x)|,1\}$ and
 $M_{1}=\sup_{|t-\underline{\alpha}|<1}|\sigma'(t)|$, we have
 \begin{align*}
  \left\|
  (\sigma(\alpha)-\sigma(\tilde{\alpha}))\frac{u_{x}u_{xx}}{\left(\sqrt{1+u_{x}^{2}}\right)^{5}}
  \right\|_{L^{2}(0,1)}
  &\leq
  \left|
  \int_{\tilde{\alpha}}^{\alpha}
  |\sigma'(t)|
  dt\right|\|u_{xx}\|_{L^{2}(0,1)}\\
   &\leq M_{1}|\alpha-\tilde{\alpha}|\|u\|_{H_{\text{per.ave}}^{2}(0,1)}
   \\
   &\to 0\quad
  \text{as}\ 
  (u,\alpha)\to(\tilde{u},\tilde{\alpha}).
 \end{align*}
 Similarly using $\sqrt{1+u_x^2(x)}\geq \max\{|u_{x}(x)|,1\}$, we have
 \begin{align*}
  \left\|
  \sigma(\tilde{\alpha})(u_{xx}-\tilde{u}_{xx})\frac{u_{x}}{\left(\sqrt{1+u_{x}^{2}}\right)^{5}}
  \right\|_{L^{2}(0,1)}
  &\leq
  |\sigma(\tilde{\alpha})|\|u-\tilde{u}\|_{H^{2}_{\mathrm{per.ave}}(0,1)}
  \to 0\quad
 \end{align*}
 as $(u,\alpha)\to(\tilde{u},\tilde{\alpha})$.
 Next, we use Lemma \ref{lem:LS.Vol_Aux_Inequality} and
\begin{equation*}
 \left|
  \frac{u_{x}(x)}{\left(\sqrt{1+u_{x}^{2}(x)}\right)^{5}}
  -
  \frac{\tilde{u}_{x}(x)}{\left(\sqrt{1+\tilde{u}_{x}^{2}(x)}\right)^{5}}
 \right|
 \leq
 15
 |u_{x}(x)-\tilde{u}_{x}(x)|.
\end{equation*}
 Thus, we compute
 \begin{align*}
  &\quad\left\|\sigma(\tilde{\alpha})\tilde{u}_{xx}(x)\left(\frac{1}{\left(\sqrt{1+u_{x}^{2}(x)}\right)^{5}}-\frac{1}{\left(\sqrt{1+\tilde{u}_{x}^{2}(x)}\right)^{5}}\right)\right\|_{L^{2}(0,1)}\\
  &\leq
  |\sigma(\tilde{\alpha})|
  \|\tilde{u}_{xx}\|_{L^{2}(0,1)}
  \sup_{x\in(0,1)}\left|\frac{1}{\left(\sqrt{1+u_{x}^{2}(x)}\right)^{5}}-\frac{1}{\left(\sqrt{1+\tilde{u}_{x}^{2}(x)}\right)^{5}}\right|
  \\
  &\leq
  15|\sigma(\tilde{\alpha})|\|\tilde{u}_{xx}\|_{L^{2}(0,1)}
  \sup_{x\in(0,1)}|u_{x}(x)-\tilde{u}_{x}(x)|.
 \end{align*}
 From the Sobolev inequality \eqref{eq:LS.Sobolev_Embedding_Y}, we have
 \begin{multline*}
  |\sigma(\tilde{\alpha})|\|\tilde{u}_{xx}\|_{L^{2}(0,1)}
  \sup_{x\in(0,1)}|u_{x}(x)-\tilde{u}_{x}(x)|
  \\
  \leq
  \sqrt{2}
  |\sigma(\tilde{\alpha})|\|\tilde{u}\|_{H^{2}_{\mathrm{per.ave}}(0,1)}
  \|u-\tilde{u}\|_{H^{2}_{\mathrm{per.ave}}(0,1)}
  \rightarrow0
 \end{multline*}
 as $(u,\alpha)\rightarrow(\tilde{u},\tilde{\alpha})$. Therefore, we have
 \begin{equation}
  \label{eq:LS.conti-Y-4}
 \left\|
   \sigma(\alpha)
   \frac{u_{x}u_{xx}}{\left(\sqrt{1+u_{x}^{2}}\right)^{5}}
   -
   \sigma(\tilde{\alpha})
   \frac{\tilde{u}_{x}\tilde{u}_{xx}}{\left(\sqrt{1+\tilde{u}_{x}^{2}}\right)^{5}}
  \right\|_{L^2(0,1)}
  \rightarrow0
  \quad
  \text{as}\
  (u,\alpha)\rightarrow(\tilde{u},\tilde{\alpha}).
 \end{equation}

 Next, we compute the fifth term of the right-hand side of
 \eqref{eq:LS.Freschet_OperatorNorm}. Observing that
 \begin{align*}
  &\quad
  \sigma''(\alpha)\int_{0}^{1}\sqrt{1+u_{x}^{2}}\,dx-\sigma''(\tilde{\alpha})\int_{0}^{1}\sqrt{1+\tilde{u}_{x}^{2}}\,dx\\
  &=(\sigma''(\alpha)-\sigma''(\tilde{\alpha}))\int_{0}^{1}\sqrt{1+u_{x}^{2}}\,dx
  \\
  &\quad
  +\sigma''(\tilde{\alpha})\int_{0}^{1}\left(\sqrt{1+u_{x}^{2}(x)}-\sqrt{1+\tilde{u}_{x}^{2}(x)}\right)dx,
 \end{align*}
 we let $M_{3}=\sup_{|t-\underline{\alpha}|\leq 1}|\sigma'''(t)|$. Using
 the H\"older inequality, we deduce
 \begin{align*}
  \left|
  (\sigma''(\alpha)-\sigma''(\tilde{\alpha}))
  \int_{0}^{1}\sqrt{1+u_{x}^{2}}\,dx\right|
  &\leq\left|\int_{\tilde{\alpha}}^{\alpha}|\sigma'''(t)|dt
  \right|
  \left(
  \int_{0}^{1}\left(1+u_{x}^{2}(x)\right)dx
  \right)^{\frac{1}{2}}
  \\
  &\leq
  M_{3}|\alpha-\tilde{\alpha}|
  \left(
  1+\|u_{x}\|_{L^{2}(0,1)}
  \right)
  \\
  &\leq M_{3}|\alpha-\tilde{\alpha}|\left(1+\|u\|_{H^{2}_{\mathrm{per.ave}}(0,1)}\right)
  \rightarrow0
  \quad
 \end{align*}
 as $(u,\alpha) \rightarrow (\tilde{u},\tilde{\alpha})$. Next, using
  Lemma \ref{lem:LS.Vol_Aux_Inequality},
 \begin{equation*}
  \left|\sigma''(\tilde{\alpha})\int_{0}^{1}\left(\sqrt{1+u_{x}^{2}(x)}-\sqrt{1+\tilde{u}_{x}^{2}(x)}\right)dx\right|
   \leq
   |\sigma''(\tilde{\alpha})|
   \int_0^1
   |u_{x}(x)-\tilde{u}_{x}(x)|
   \,dx.
 \end{equation*}
 By the H\"older inequality, we obtain
 \begin{equation*}
  \int_0^1
  |u_{x}(x)-\tilde{u}_{x}(x)|
  \,dx
  \leq
  \|u-\tilde{u}\|_{H^2_{\mathrm{per.ave}}(0,1)}
 \end{equation*}
 hence
 \begin{multline*}
  \left|\sigma''(\tilde{\alpha})\int_{0}^{1}\left(\sqrt{1+u_{x}^{2}(x)}-\sqrt{1+\tilde{u}_{x}^{2}(x)}\right)dx\right|
  \\
  \leq
   |\sigma''(\tilde{\alpha})|
   \|u-\tilde{u}\|_{H^2_{\mathrm{per.ave}}(0,1)}
   \rightarrow0
 \end{multline*}
 as $(u,\alpha) \rightarrow
 (\tilde{u},\tilde{\alpha})$.
 Therefore, we obtain
 \begin{equation}
  \label{eq:LS.conti-Y-5}
   \left|
    \int_{0}^{1}
    \left\{
     \sigma''(\alpha)
     \sqrt{1+u_{x}^{2}(x)}
     -
     \sigma''(\tilde{\alpha})
     \sqrt{1+\tilde{u}_{x}^{2}(x)}
    \right\}
    dx
   \right|
   \rightarrow
   0
   \quad
   \text{as}
   \
   (u,\alpha)
  \rightarrow
  (\tilde{u},\tilde{\alpha}).
 \end{equation}

 Finally, we compute the sixth term of the right-hand side of
 \eqref{eq:LS.Freschet_OperatorNorm}. We consider the following
 transform
 \begin{align*}
  &\quad
  \sigma'(\alpha)
  \int_{0}^{1}
  \frac{u_{x}(x)}{\sqrt{1+u_{x}^{2}(x)}}dx
  -
  \sigma'(\tilde{\alpha})
  \int_{0}^{1}\frac{\tilde{u}_{x}(x)}{\sqrt{1+\tilde{u}_{x}^{2}(x)}}dx\\
  &=
  (\sigma'(\alpha)-\sigma'(\tilde{\alpha}))
  \int_{0}^{1}\frac{u_{x}(x)}{\sqrt{1+u_{x}^{2}(x)}}dx
  \\
  &\quad
  +
  \sigma'(\tilde{\alpha})
  \int_{0}^{1}
  \left(
  \frac{u_{x}(x)}{\sqrt{1+u_{x}^{2}(x)}}
  -
  \frac{\tilde{u}_{x}(x)}{\sqrt{1+\tilde{u}_{x}^{2}(x)}}
  \right)dx.
 \end{align*}
 Recall $M_{2}=\sup_{|t-\underline{\alpha}|\leq 1}|\sigma''(t)|$. Since
 $\sqrt{1+u_x^2(x)}\geq|u_x(x)|$, we have
 \begin{equation*}
  \left|(\sigma'(\alpha)-\sigma'(\tilde{\alpha}))\int_{0}^{1}\frac{u_{x}(x)}{\sqrt{1+u_{x}^{2}(x)}}dx\right|
   \leq\left|\int_{\tilde{\alpha}}^{\alpha}|\sigma''(t)|dt\right|
   \leq M_{2}|\alpha-\tilde{\alpha}|
   \to 0
 \end{equation*}
 as $(u,\alpha)\to(\tilde{u},\tilde{\alpha})$. Next, applying Lemma
 \ref{lem:LS.Vol_Aux_Inequality} and the H\"older inequality, we have
 \begin{equation*}
  \begin{split}
   \left|\sigma'(\tilde{\alpha})\int_{0}^{1}\left(\frac{u_{x}(x)}{\sqrt{1+u_{x}^{2}(x)}}-\frac{\tilde{u}_{x}(x)}{\sqrt{1+\tilde{u}_{x}^{2}(x)}}\right)dx\right|
   &\leq |\sigma'(\tilde{\alpha})|
   \int_0^1|u_{x}(x)-\tilde{u}_{x}(x)|\,dx
   \\
   &\leq |\sigma'(\tilde{\alpha})|
   \|u-\tilde{u}\|_{H^2_{\mathrm{per.ave}}(0,1)}
   \rightarrow0
  \end{split}
 \end{equation*}
 as $(u,\alpha)\rightarrow(\tilde{u},\tilde{\alpha})$. Thus, we obtain
 \begin{equation}
  \label{eq:LS.conti-Y-6}
\left|\int_{0}^{1}\left\{\sigma'(\alpha)\frac{u_{x}(x)}{\sqrt{1+u_{x}^{2}(x)}}-\sigma'(\tilde{\alpha})\frac{\tilde{u}_{x}(x)}{\sqrt{1+\tilde{u}_{x}^{2}(x)}}\right\}dx\right|
  \to 0
  \quad
  \text{as}\ 
  (u,\alpha)\rightarrow(\tilde{u},\tilde{\alpha}).
 \end{equation}
 
 Combining all estimate \eqref{eq:LS.conti-Y-1},
 \eqref{eq:LS.conti-Y-2}, \eqref{eq:LS.conti-Y-3},
 \eqref{eq:LS.conti-Y-4}, \eqref{eq:LS.conti-Y-5},
 \eqref{eq:LS.conti-Y-6} to \eqref{eq:LS.Freschet_OperatorNorm}, we
 obtain
 \begin{equation*}
  \|
   D\dot{E}((u,\alpha),\cdot)
   -
   D\dot{E}((\tilde{u},\tilde{\alpha}),\cdot)
   \|_{\mathscr{L}(Y)}
   \rightarrow
   0
 \end{equation*}
 as $(u,\alpha)\rightarrow(\tilde{u},\tilde{\alpha})$, which means
 $O\ni
 (u,\alpha)\rightarrow D\dot{E}((u,\alpha),\cdot)\in \mathscr{L}(Y)$ is
 continuous on $Y$.
\end{proof}

Now we are in a position to show Fr\'{e}chet differentialbility of
$\dot{E}$ on $Y$.

\begin{proof}[Proof of Proposition \ref{prop:LS.Frechet_Y}]
 Applying Lemma \ref{lem:LS.Gateau_Frechet} with Lemmas \ref{app:lem1}
 and \ref{app:lem2}, $\dot{E}$ is Fr\'{e}chet differentiable at
 $(\underline{u},\underline{\alpha})\in Y$. For any $(h,\beta)\in Y$,
 \begin{equation*}
 (\dot{E})'(\underline{u},\underline{\alpha})(h,\beta)
  =
  D\dot{E}((\underline{u},\underline{\alpha}),(h,\beta))
 \end{equation*}
hence $(\dot{E})'$ is continuous on $\mathscr{L}(Y)$ from Lemma
\ref{app:lem2}. Therefore $\dot{E}$ is continuously Fr\'{e}chet
differentiable on $Y$ and Proposition \ref{prop:LS.Frechet_Y} is proved.
\end{proof}

Finally, we show the main theorem in this section.

\begin{proof}[Proof of Theorem \ref{thm:Lojasiewicz-Simon}]
 Theorem \ref{thm:Lojasiewicz-Simon} can be deduced from Theorem
 \ref{thm:LS.LS}, together with Propositions \ref{prop:LS.Conti_diff_X},
 \ref{prop:LS.L_Fredholm}, \ref{prop:LS.Assumption_dense},
 \ref{prop:LS.CriticalPt_in_Y}, \ref{prop:LS.inverse_image_L},
 \ref{prop:LS.Image_on_Y}, \ref{prop:LS.Frechet_Y}. We will show that
 $E$ is analytic on the critical manifold $S$.

 The critical manifold $S$ is given by
 \begin{equation*}
  S
   =
   \{(u,\alpha)\in Y:L(\dot{E}(u,\alpha))=0\}
   =
   \{(u,\alpha)\in Y:\dot{E}(u,\alpha)\in \Ker(L)\}.
 \end{equation*}
 From Propositions \ref{prop:LS.L_Fredholm} and
 \ref{prop:LS.Conti_diff_X}, 
 we have
 for $(u,\alpha)\in Y$ with
 $\dot{E}(u,\alpha)\in\Ker(L)$ that, 
 \begin{equation*}
  \sigma(\alpha)
   \left\{
    \int_{0}^{x}\frac{u_{x}(y)}{\sqrt{1+u_{x}^{2}(y)}}dy
    -
    \int_{0}^{1}\int_{0}^{x}\frac{u_{x}(y)}{\sqrt{1+u_{x}^{2}(y)}}dydx
   \right\}
   =0.
 \end{equation*}
 Taking a derivative with respect to $x$, we have
 \begin{equation*}
  \sigma(\alpha)\frac{u_{x}(x)}{\sqrt{1+u_{x}^{2}(x)}}=0.
 \end{equation*}
 From \eqref{eq:1.Assumption_positivity}, we have $u_x=0$ so $u$ is a
 constant function on $(0,1)$. Since $u\in H^2_{\mathrm{per.ave}}(0,1)$,
 we have $u\equiv0$. Again using Proposition \ref{prop:LS.L_Fredholm}
 and \ref{prop:LS.Conti_diff_X}, we obtain
 \begin{equation*}
  S=
   \begin{cases}
    \{0\}\times \R,
    & 
    \text{if}\ \sigma''(\overline{\alpha})=0,
    \\
    \{0\}\times \{\alpha\in \R:\sigma'(\alpha)=0\}, 
    & 
    \text{if}\ \sigma''(\overline{\alpha})\neq0. \\
   \end{cases}
 \end{equation*}
 Thus, on the critical manifold $S$, the grain boundary energy $E$ is
 written as
 \begin{equation*}
  E[u,\alpha]=\sigma(\alpha),\qquad (u,\alpha)\in S.
 \end{equation*}
 Since $\sigma$ is an analytic function on $\R$, $E$ is also analytic on
 $S$. Therefore, we can apply Theorem \ref{thm:LS.LS} so there exist a
 neiborhood of $(\overline{u},\overline{\alpha})$ in $Y$ and constants
 $0< \theta\leq\frac{1}{2}$, $C>0$ such that
 \begin{equation*}
  \|\dot{E}(u,\alpha)\|_{Y}
   \ge
   C|E[u,\alpha]-E[\overline{u},\overline{\alpha}]|^{1-\theta}
 \end{equation*}
 for all $(u,\alpha)\in U$.
\end{proof}

\section{Applications}
\label{sec:Application}

In this section, we discuss some applications for the
\L{}ojasiewicz-Simon inequality \eqref{eq:LS.LS_for_GBE}.

\subsection{Long-time asymptotic behavior of the evolution problem}

Let us consider the following problem
\begin{equation}
 \label{eq:3.1.GeomEvolEq}
  \left\{
   \begin{aligned}
    \frac{u_{t}(x,t)}{\sqrt{1+u_{x}^{2}(x,t)}}
    &=
    \mu\sigma(\alpha(t))\left(\frac{u_{x}(x,t)}{\sqrt{1+u_{x}^{2}(x,t)}}\right)_{x},&\quad
    &x\in\mathbb{T},\ t>0,
    \\
    \alpha_t(t)
    &=-\gamma\sigma'(\alpha(t))\int_{0}^{1}\sqrt{1+u_{x}^{2}(x,t)}dx,
    &
    &t>0,
    \\
    u(x,0)&=u_0(x),&
    &x\in\mathbb{T},
    \\
    \alpha(0)
    &=
    \alpha_0.
   \end{aligned}
  \right.
\end{equation}
Here $\mathbb{T}=\R/\Z$ is a torus,
$u=u(x,t):\mathbb{T}\times[0,\infty)\rightarrow\R$ and
$\alpha:[0,\infty)\rightarrow\R$ are unknown functions,
$\sigma:\R\rightarrow\R$ is a given function,
$u_0:\mathbb{T}\rightarrow\R$, $\alpha_0\in\R$ are given initial data,
and $\mu,\gamma>0$ are positive constants. This problem was studied by
\cite{MR4292952}, motivated by the work \cite{MR4263432} to investigate
the evolution of grain boundaries with misorientation effects. From the
mathematical model, the graph $\Gamma_t:=\{u(x,t):x\in\mathbb{T}\}$
represents the grain boundary, and $\alpha$ represents the
misorientation angle, which is the difference between the lattice of the
two grains on either side of the grain boundary $\Gamma_t$.

Now, we introduce the previous results in \cite{MR4292952}.

\begin{proposition}[{\cite[Theorem 5.1]{MR4292952}}]
 Assume that $u_0$ is a Lipschitz function on $\mathbb{T}$ with a
 Lipschitz constant $M > 0$, $\beta\in(0,1)$, $\alpha_0\in\R$ and
 $\sigma\in C^1(\R)$ satisfies \eqref{eq:1.Assumption_positivity} and
 \eqref{eq:1.ConvexityAssumption}. Moreover, there exists $L > 0$ such that
 $|\sigma'(a)-\sigma'(b)|\leq L|a-b|$ for any $a,b\in\R$. Then, there
 exists a constant $u_\infty$ such that
 $\|u_\infty-u\|_{C^2(\mathbb{T})}$ goes to $0$ exponentially. In
 addition, the curvature $\kappa$ also goes to $0$ uniformly and
 exponentially on $\mathbb{T}$.
\end{proposition}

To prove asymptotics of $u$, we derive higher-order energy estimates for
\eqref{eq:3.1.GeomEvolEq}. The assumption \eqref{eq:1.ConvexityAssumption}
was needed to derive the maximum principle for $\alpha$. Recalling the
model for the evolution of grain boundaries on the plane, the grain
boundary energy function $\sigma$ should be periodic function with
period $\pi/2$. This is because the misorientation is considered to be
invariant under a 90-degree rotation of the crystal lattice. The
periodicity is incompatible with assumption
\eqref{eq:1.ConvexityAssumption}. Thus, it is interesting to show long-time
asymptotics of $u$ without assuming \eqref{eq:1.ConvexityAssumption}.

Using the \L{}ojasiewicz-Simon inequality \eqref{eq:LS.LS_for_GBE}, we
can relax the assumption \eqref{eq:1.ConvexityAssumption} and show
asymptotic behavior of solutions of \eqref{eq:3.1.GeomEvolEq}.

\begin{theorem}
 \label{thm:GBM_stability}
 Let $\sigma$ be an analytic function on $\R$ satisfying the positivity
 \eqref{eq:1.Assumption_positivity}. Assume $(u_0,\alpha_0)\in
 H_{\mathrm{per.ave}}^{2}(0,1)\times\R$. Let $(u,\alpha)$ be a
 time-global classical solution of \eqref{eq:3.1.GeomEvolEq} associated
 the initial data $(u_0,\alpha_0)$. Assume $(u(x,t),\alpha(t))$
 sufficiently close to an equilibrium point $(0,\overline{\alpha})\in
 H^2_{\mathrm{per.ave}}(0,1)\times\R$ in
 $H^2_{\mathrm{per.ave}}(0,1)\times\R$ for any $t>0$. Then, there exists
 a positive constant $\Cl{const:3.stability}>0$ depending only
 on $\mu$, $\gamma$, $\alpha_0$,
 $\|u_0\|_{H_{\mathrm{per.ave}}^{2}(0,1)}$, $\theta$,
 $\Cr{const:LS_GBE}$, where $\theta$, $\Cr{const:LS_GBE}$ are taken in
 Theorem \ref{thm:Lojasiewicz-Simon} (The \L{}ojasiewicz-Simon
 inequality), such that for $0<t<s<\infty$
 \begin{equation}
  \label{eq:3.stability_Estimate}
  |\alpha(t)-\alpha(s)|,\ 
   \|u(t)-u(s)\|_{L^2(0,1)}
   \leq
   \Cr{const:3.stability}
   |E[u,\alpha](t)
   -
   E[0,\overline{\alpha}]|^\theta.
 \end{equation}
\end{theorem}

From Theorem \ref{thm:GBM_stability}, we immediately obtain the
following corollary.
\begin{corollary}
 In addition to the assumption in Theorem \ref{thm:GBM_stability},
 assume there is a sequence $\{t_n\}_{n=1}^\infty$ such that
 $t_n\rightarrow\infty$ and
 \begin{equation*}
  (u(x,t_n),\alpha(t_n))\rightarrow (0,\overline{\alpha})
   \quad
   \text{in}\
   L^2(0,1)\times\R.
 \end{equation*}
 Then $(u,\alpha)$ converges to $(0,\overline{\alpha})$ in
 $L^2(0,1)\times\R$ as $t\rightarrow\infty$ and
 \begin{equation}
  \label{eq:3.stability}
  |\alpha(t)-\overline{\alpha}|,\ 
   \|u(t)\|_{L^2(0,1)}
   \leq
   \Cr{const:3.stability}
   |E[u,\alpha](t)
   -
   E[0,\overline{\alpha}]|^\theta
 \end{equation}
 for $t>0$.
\end{corollary}

\begin{proof}
 Taking $s=t_n$ and letting $n\rightarrow\infty$ in
 \eqref{eq:3.stability_Estimate}, we obtain \eqref{eq:3.stability}.
\end{proof}

To prove Theorem \ref{thm:GBM_stability}, we use the following a priori
gradient estimate.

\begin{proposition}[{\cite[{Theorem 4.2}]{MR4292952}}]
 Let $(u,\alpha)$ be a solution of \eqref{eq:3.1.GeomEvolEq} and
 let $v=\sqrt{1+u_x^2}$. Assume \eqref{eq:1.GBE}. Then, for all
 $0<x_0<1$ and $t_0>0$,
 \begin{equation}
  \label{eq:3.Gradient_Estimate}
  v(x_0,t_0)
   \leq
   \frac{\sigma(\alpha(0))}{\Cr{const:GBE_positivity}}
   \sup_{0<x<1}v^2(x,0).
 \end{equation}
\end{proposition}

\begin{proof}[Proof of Theorem \ref{thm:GBM_stability}]
 Using \eqref{eq:3.1.GeomEvolEq}, we transform the energy dissipation
 law as
 \begin{equation}
  \label{eq:3.EnergyDissipationLaw}
   \begin{split}
    \frac{d}{dt}E[u,\alpha](t)
    &=
    -\frac{1}{\gamma}|\alpha_t(t)|^{2}
    -
    \frac{1}{\mu}
    \int_0^1
    \left|
    \frac{u_{t}(x,t)}{\sqrt{1+u_{x}^{2}(x,t)}}
    \right|^2
    \sqrt{1+u_{x}^{2}(x,t)}dx
    \\
    &=
    -
    \frac{1}{\gamma}|\alpha_t(t)|^{2}
    -
    \frac{1}{\mu}
    \left\|
    \frac{u_t(t)}{\sqrt{1+u_x^2(t)}}
    \right\|_{L^2(0,1\,;\,\sqrt{1+u_{x}^{2}(t)}dx)}^2
    \\
    &=
    -
    \left(
    \frac{1}{\gamma}|\alpha_t(t)|^{2}
    +
    \frac{1}{\mu}
    \left\|
    \frac{u_t(t)}{\sqrt{1+u_x^2(t)}}
    \right\|_{L^2(0,1\,;\,\sqrt{1+u_{x}^{2}(t)}dx)}^2
    \right)^{\frac{1}{2}}
    \\
    &\quad
    \times
    \biggl(
    \gamma
    \left|\sigma'(\alpha(t))
    \int_0^1
    \sqrt{1+u_x^2(x,t)}\,dx
    \right|^{2}
    \\
    &\quad
    +
    \mu
    \left\|
    \sigma(\alpha(t))
    \left(
    \frac{u_x(t)}{\sqrt{1+u_x^2(t)}}
    \right)_x
    \right\|_{L^2(0,1\,;\,\sqrt{1+u_{x}^{2}(t)}dx)}^2
    \biggr)^{\frac{1}{2}},
   \end{split}
 \end{equation}
 where $\|\cdot\|_{L^2(0,1\,;\,\sqrt{1+u_{x}^{2}(t)}dx)}$ is the
 weighted $L^2$ norm, namely for a Lebesgue measurable function $f$ on
 $(0,1)$,
 \begin{equation*}
  \|f\|_{L^2(0,1\,;\,\sqrt{1+u_{x}^{2}(t)}dx)}
   =
   \left(
    \int_0^1
    |f(x)|^2
    \sqrt{1+u_x^2(x,t)}
    \,dx
   \right)^{\frac{1}{2}}.
 \end{equation*}
 From the \L{}ojasiewicz-Simon inequality \eqref{eq:LS.LS_for_GBE}, we have
 \begin{multline}
  \label{eq:3.LS_Inequality_GBE}
  |\sigma(\alpha)|^2
  \left(
  \left\|
  \frac{u_x(t)}{\sqrt{1+u_x^2(t)}}
  \right\|_{L^2(0,1)}^2
  +
  \left\|
  \left(
  \frac{u_x(t)}{\sqrt{1+u_x^2(t)}}
  \right)_x
  \right\|_{L^2(0,1)}^2
  \right)
  \\  
  +
  |\sigma'(\alpha)|^2
  \left(
  \int_0^1\sqrt{1+u_x^2(x,t)}\,dx\right)^2
  \geq
  \Cr{const:LS_GBE}^2|E[u,\alpha](t)-E[0,\overline{\alpha}]|^{2(1-\theta)}.
 \end{multline}
 Note that $u$ is a periodic function on $(0,1)$, so for any $t>0$,
 there is $0<x_1<1$ such that $u_x(x_1,t)=0$. Then, for any $0<x<1$,
 \begin{equation*}
  \begin{split}
   \left|
   \frac{u_x(x,t)}{1+u^2_x(x,t)}
   \right|
   =
   \left|
    \frac{u_x(x,t)}{1+u^2_x(x,t)}
    -
    \frac{u_x(x_1,t)}{1+u^2_x(x_1,t)}
   \right|
   &\leq
   \left|\int_{x_1}^x
   \left|
    \left(
     \frac{u_x(y,t)}{1+u^2_x(y,t)}
    \right)_y
   \right|
    \,dy
   \right|
   \\
   &\leq
   \left\|
    \left(
     \frac{u_x}{1+u^2_x}
    \right)_x
   \right\|_{L^2(0,1)},
  \end{split} 
\end{equation*}
 hence
 \eqref{eq:3.LS_Inequality_GBE}
 turns into 
 \begin{multline*}
  2|\sigma(\alpha)|^2
   \left\|
    \left(
     \frac{u_x(t)}{\sqrt{1+u_x^2(t)}}
    \right)_x
   \right\|_{L^2(0,1)}^2
   +
   |\sigma'(\alpha)|^2
   \left(
    \int_0^1\sqrt{1+u_x^2(x,t)}\,dx\right)^2
   \\
  \geq
   \Cr{const:LS_GBE}^2
   |E[u,\alpha](t)-E[0,\overline{\alpha}]|^{2(1-\theta)}.
 \end{multline*}
 Then, applying this inequality to the energy dissipation
 law~\eqref{eq:3.EnergyDissipationLaw} with $\sqrt{1+u_x^2}\geq1$, we
 have
 \begin{equation*}
  \begin{split}
   &\quad
   \frac{d}{dt}E[u,\alpha](t)
   \\
   &\leq
   -
   \frac{\sqrt{\min\{\gamma,\mu\}}}{2}
   \sqrt{
   \frac{1}{\gamma}|\alpha_t(t)|^{2}
   +
   \frac{1}{\mu}
   \left\|
   \frac{u_t(t)}{\sqrt{1+u_x^2(t)}}
   \right\|_{L^2(0,1\,;\,\sqrt{1+u_{x}^{2}(t)}dx)}^2
   }
   \\
   &\quad
   \times
   \sqrt{
   2
   \left|\sigma'(\alpha(t))
   \int_0^1
   \sqrt{1+u_x^2(x,t)}\,dx
   \right|^{2}
   +
   2
   \left\|
   \sigma(\alpha(t))
   \left(
   \frac{u_x(t)}{\sqrt{1+u_x^2(t)}}
   \right)_x
   \right\|_{L^2(0,1)}^2
   }
   \\
   &\leq
   -
   \Cr{const:3.LS_with_mobility}
   |E[u,\alpha](t)-E[0,\overline{\alpha}]|^{1-\theta}
   \sqrt{
   \frac{1}{\gamma}|\alpha_t(t)|^{2}
   +
   \frac{1}{\mu}
   \left\|
   \frac{u_t(t)}{\sqrt{1+u_x^2(t)}}
   \right\|_{L^2(0,1\,;\,\sqrt{1+u_{x}^{2}(t)}dx)}^2
   }
  \end{split}
 \end{equation*} 
 where
 $\Cl{const:3.LS_with_mobility}=\Cr{const:LS_GBE}\frac{\sqrt{\min\{\gamma,\mu\}}}{2}>0$.

 From the gradient estimate \eqref{eq:3.Gradient_Estimate} and the
 Sobolev inequality \eqref{eq:LS.Sobolev_Embedding_Y}, we have
 \begin{equation*}
  \begin{split}
   \sqrt{1+u_x^2(x,t)}
   &\leq
   \frac{\sigma(\alpha(0))}{\Cr{const:GBE_positivity}}
   \sup_{0<x<1}
   (1+(u_0)^2_x(x))
   \\
   &\leq
   \frac{\sigma(\alpha(0))}{\Cr{const:GBE_positivity}}
   (1+\|u_0\|_{H_{\mathrm{per.ave}}^{2}(0,1)})
   =:
   \Cl{const:3.GradientBounds}.
  \end{split} 
\end{equation*}
 Then, the normal velocity $\frac{u_t}{\sqrt{1+u_x^2}}$ can be estimated
 as
 \begin{equation*}
  \left\|
   \frac{u_t(t)}{\sqrt{1+u_x^2(t)}}
  \right\|_{L^2(0,1\,;\,\sqrt{1+u_{x}^{2}(t)}dx)}^2
  \geq
  \frac{1}{\Cr{const:3.GradientBounds}}
  \|u_t(t)\|_{L^2(0,1)}^2.
 \end{equation*}
 Therefore, we obtain from the energy dissipation law that
  \begin{multline}
   \label{eq:3.EnergyDissipationLaw_with_LS}
   \frac{d}{dt}E[u,\alpha](t)
   \\
   \leq
   -
   \Cr{const:3.LS_with_mobility}
   |E[u,\alpha](t)-E[0,\overline{\alpha}]|^{1-\theta}
   \sqrt{
   \frac{1}{\gamma}|\alpha_t(t)|^{2}
   +
   \frac{1}{\mu\Cr{const:3.GradientBounds}}
   \|u_t\|_{L^2(0,1)}^2
   }.
  \end{multline} 
 Now, we consider the derivative of
 $|E[u,\alpha](\tau)-E[0,\overline{\alpha}]|^{\theta}$ in terms of
 $\tau$. By \eqref{eq:3.EnergyDissipationLaw_with_LS}, we have
 \begin{equation*}
  \begin{split}
   -\frac{d}{d\tau}
   |E[u,\alpha](\tau)-E[0,\overline{\alpha}]|^{\theta}
   &=
   -\theta
   |E[u,\alpha](\tau)-E[0,\overline{\alpha}]|^{\theta-1}
   \frac{d}{d\tau}E[u,\alpha](\tau)
   \\
   &\geq
   \Cr{const:3.LS_with_mobility}
   \theta
   \sqrt{
   \frac{1}{\gamma}|\alpha_t(\tau)|^{2}
   +
   \frac{1}{\mu\Cr{const:3.GradientBounds}}
   \|u_t(\tau)\|_{L^2(0,1)}^2
   }.
  \end{split} 
 \end{equation*}
 Then, for $0<t<s<\infty$
 \begin{equation*}
  \begin{split}
   |\alpha(t)-\alpha(s)|
   &\leq
   \int_{t}^{s}
   |\alpha_t(\tau)|
   \,d\tau
   \\
   &\leq
   -
   \frac{\sqrt{\gamma}}{\Cr{const:3.LS_with_mobility}\theta}
   \int^{s}_{t}
   \frac{d}{d\tau}
   |E[u,\alpha](\tau)-E[0,\overline{\alpha}]|^{\theta}
   \,d\tau
   \\
   &\leq
   \frac{\sqrt{\gamma}}{\Cr{const:3.LS_with_mobility}\theta}
   |E[u,\alpha](t)-E[0,\overline{\alpha}]|^{\theta}
  \end{split} 
\end{equation*}
 and
 \begin{equation*}
  \begin{split}
   \|u(t)-u(s)\|_{L^2(0,1)}
   &\leq
   \int^{s}_{t}
   \|u_t(\tau)\|_{L^2(0,1)}
   \,d\tau
   \\
   &\leq
   -
   \frac{\sqrt{\mu \Cr{const:3.GradientBounds}}}{\Cr{const:3.LS_with_mobility}\theta}
   \int^{s}_{t}
   \frac{d}{d\tau}
   |E[u,\alpha](\tau)-E[0,\overline{\alpha}]|^{\theta}
   \,d\tau
   \\
   &\leq
   \frac{\sqrt{\mu \Cr{const:3.GradientBounds}}}{\Cr{const:3.LS_with_mobility}\theta}
   |E[u,\alpha](t)-E[0,\overline{\alpha}]|^{\theta}
  \end{split} 
\end{equation*}
 Thus, \eqref{eq:3.stability_Estimate} is deduced by taking
 $\Cr{const:3.stability}=\frac{\max\{\sqrt{\gamma}, \sqrt{\mu
 \Cr{const:3.GradientBounds}}\}}{\Cr{const:3.LS_with_mobility}\theta}$.
\end{proof}

\begin{remark}
 We emphasize that the assumption \eqref{eq:1.ConvexityAssumption} does
 not need to obtain the long-time asymptotic behavior of
 $\Gamma_t$. Furthermore, this argument may be extended to more
 complicated problems, such as involving other state variables. For
 example, for mathematical modelling of 2D grain boundary motion, it is
 not enough to study only misorientation
 (cf.\cite{doi:10.1016/j.actamat.2017.02.010,doi:10.1016/j.jmps.2018.05.001}). This
 research will give a fundamental contribution to studying the long-time
 asymptotic behavior of a mathematical model for grain boundary
 motion. On the other hand, as in \cite{MR4292952}, the energy method
 gives a finer asymptotics for $\Gamma_t$. Note in \cite{MR4292952} that
 they showed uniform convergence of the curvature of $\Gamma_t$ as
 $t\rightarrow\infty$.
\end{remark}

\begin{remark}
 In Theorem \ref{thm:GBM_stability}, we note the closeness of
 $(u,\alpha)$ to $(0,\overline{\alpha})$ in
 $H^2_{\mathrm{per.ave}(0,1)}\times\R$. We assume that there is a
 sequence $\{t_n\}_{n=1}^\infty$ such that $t_n\rightarrow\infty$ and
 $\alpha(t_n)\rightarrow\overline{\alpha}$ as $n\rightarrow\infty$. This
 assumption generally comes from the stability of
 $(0,\overline{\alpha})$ in the problem \eqref{eq:3.1.GeomEvolEq}.
 Furthermore, assume $E[u,\alpha](t)\rightarrow E[0,\overline{\alpha}]$
 as $t\rightarrow\infty$. This assumption is usually derived from the
 energy dissipation law~\eqref{eq:1.EnergyLaw}, or monotonicity of the
 energy $E$ along with the solution of \eqref{eq:3.1.GeomEvolEq}.

 We recall \cite[{(5.3)}]{MR4292952}, or differentiate
 \eqref{eq:3.1.GeomEvolEq} with respect to $x$ that
 \begin{equation*}
  u_{xt}
   =
   \frac{\mu\sigma(\alpha(t))}{1+u_x^2}u_{xxx}
   -
   \frac{2\mu\sigma(\alpha(t))}{(1+u_x^2)^2}u_xu_{xx}^2.
 \end{equation*}
 Multiplying $u_x$ and using $(u_x^2)_{xx}=2u_{xx}^2+2u_{x}u_{xxx}$, we
 have
 \begin{equation*}
  \begin{split}
   \left(
   \frac{1}{2}
   u_x^2\right)_{t}
   &=
   \frac{\mu\sigma(\alpha(t))}{1+u_x^2}
   \left(
   \frac{1}{2}
   u_x^2\right)_{xx}
   -
   \frac{\mu\sigma(\alpha(t))}{1+u_x^2}
   u_{xx}^2 
   -
   \frac{2\mu\sigma(\alpha(t))}{(1+u_x^2)^2}
   u_x^2u^2_{xx}
   \\
   &\leq
   \frac{\mu\sigma(\alpha(t))}{1+u_x^2}
   \left(
   \frac{1}{2}
   u_x^2\right)_{xx}.
  \end{split} 
 \end{equation*}
 This yields that $\frac{1}{2}u_{x}^2$ is a subsolution of a parabolic
 equation, so we use the maximum principle that
 \begin{equation}
  \label{eq:3.GradientBounds_MP}
   u_x^2(x,t)\leq\sup_{x\in(0,1)}(u_0)^2_x(x)
 \end{equation}
 for $0<x<1$ and $t>0$. Then, we recall \cite[{(5.8)}]{MR4292952} and use
 the Young inequality that
 \begin{equation}
  \label{eq:3.Energy_2ndDerivative}
  \begin{split}
   &\quad
   \frac{d}{dt}
   \int_0^1
   |u_{xx}(x,t)|^2\,dx
   \\
   &=
   -2\int_0^1
   \frac{\mu\sigma(\alpha)}{1+|u_x|^2}u^2_{xxx}\,dx
   -
   8\int_0^1
   \frac{\mu\sigma(\alpha)}{(1+|u_x|^2)^2}u_xu^2_{xx}u_{xxx}\,dx \\
   &\quad
   -
   4\int_0^1\frac{\mu\sigma(\alpha)}{(1+|u_x|^2)^2}u^4_{xx} \,dx
   +
   16\int_0^1
   \frac{\mu\sigma(\alpha)}{(1+|u_x|^2)^3}u^2_{x}u_{xx}^4\,dx
   \\
   &\leq
   -\int_0^1
   \frac{\mu\sigma(\alpha)}{1+|u_x|^2}u^2_{xxx}\,dx
   -
   4\int_0^1\frac{\mu\sigma(\alpha)}{(1+|u_x|^2)^2}u^4_{xx} \,dx
   \\
   &\quad
   +
   32\int_0^1
   \frac{\mu\sigma(\alpha)}{(1+|u_x|^2)^3}u^2_{x}u_{xx}^4\,dx.
  \end{split}
 \end{equation}
 Here, we used $1+|u_x|^2\geq1$. If
 $\sup_{x\in(0,1)}(u_0^2)_x(x)<\frac{1}{8}$, then using
 \eqref{eq:3.GradientBounds_MP} and $1+|u_x|^2\geq1$ to
 \eqref{eq:3.Energy_2ndDerivative} that
 \begin{equation*}
  \frac{d}{dt}
   \int_0^1
   |u_{xx}(x,t)|^2\,dx
   \leq
   -\int_0^1
   \frac{\mu\sigma(\alpha)}{1+|u_x|^2}u^2_{xxx}\,dx.
 \end{equation*}
 Therefore, $\|u\|_{H^2_{\mathrm{per.ave}}(0,1)}$ can be controlled by
 $\|u_0\|_{H^2_{\mathrm{per.ave}}(0,1)}$. We note that the system of the
 differential equations \eqref{eq:3.1.GeomEvolEq}, which we handle, is a
 one-dimensional, periodic boundary problem, and the equations of the
 derivative of the solution are simple. Thus, we can apply the above
 energy methods.

 To control the closeness of $\alpha$ to $\overline{\alpha}$, we choose
 $r>0$ such that
 \begin{equation*}
  \{v\in H^2_{\mathrm{per.ave}}(0,1):\|v\|_{H^2_{\mathrm{per.ave}}(0,1)}<r\}
   \times
   (\overline{\alpha}-r, \overline{\alpha}+r)
   \subset U,
 \end{equation*}
 where $U$ is stated in Theorem \ref{thm:Lojasiewicz-Simon}. Next, we
 take $t_N>0$ such that $\alpha(t_N)\in (\overline{\alpha}-\frac{r}{3},
 \overline{\alpha}+\frac{r}{3})$ and $\Cr{const:3.stability}
 |E[u,\alpha](t) - E[0,\overline{\alpha}]|^\theta<\frac{r}{3}$ for any
 $t> t_N$. Then, for $t>t_N$ with $\alpha(t)\in
 (\overline{\alpha}-r,\overline{\alpha}+r)$, we can use
 \eqref{eq:3.stability_Estimate} and
 \begin{equation*}
  \begin{split}
   |\alpha(t)-\overline{\alpha}|
   &\leq
   |\alpha(t)-\alpha(t_N)|
   +
   |\alpha(t_N)-\overline{\alpha}|
   \\
   &\leq
   \Cr{const:3.stability}
   |E[u,\alpha](t) - E[0,\overline{\alpha}]|^\theta
   +
   |\alpha(t_N)-\overline{\alpha}|
   <
   \frac{2}{3}r,
  \end{split} 
 \end{equation*}
 hence $\alpha(t)\in (\overline{\alpha}-r, \overline{\alpha}+r)$ for all
 $t>t_N$. This argument can be seen in \cite[{p.37}]{MR4274456}.
\end{remark}

\subsection{A grain length estimate by the grain boundary energy}
The \L{}ojasiewicz-Simon inequality \eqref{eq:LS.LS_for_GBE}
gives a length estimate for any grain boundary, which is close to the
equilibrium state.

To state the theorem, we observe a periodic smooth function
$f:(0,1)\rightarrow\R$. By Rolle's theorem, there is $x_0\in(0,1)$ such
that $f_x(x_0)=0$. Thus, for $x\in(0,1)$
\begin{equation*}
 |f_x(x)|^2
  =
  |f_x(x)-f_x(x_0)|^2
  =
  \left|
   \int_0^1 f_{xx}(\xi)\,d\xi
  \right|^2
  \leq
  \int_0^1 |f_{xx}(\xi)|^2\,d\xi
\end{equation*}
hence we obtain
\begin{equation}
 \label{eq:leq3.Periodic_PoincareInequality}
  \int_0^1|f_x(x)|^2\,dx
  \leq
  \int_0^1|f_{xx}(x)|^2\,dx.
\end{equation}

Using \eqref{eq:leq3.Periodic_PoincareInequality}, we have the following
inequality.

\begin{theorem}
 \label{thm:3.LengthEstimate}
 Let $u:(0,1)\rightarrow\R$ be a smooth periodic function and
 $\alpha\in\R$.  Let $\overline{\alpha}\in\R$ be a critical point of
 $\sigma$, i.e. $\sigma'(\overline{\alpha})=0$. Then there exist
 $\frac{1}{2}<\gamma\leq1$, $\Cl{const:LengthEstimate}>0$, and $\delta>0$ such that if
 $\|u\|_{H_{\mathrm{per.ave}}^{2}(0,1)}<\delta$ and
 $|\alpha-\overline{\alpha}|<\delta$, then
 \begin{equation}
  \label{eq:3.Length_Estimate_by_GBE}
  \sigma(\alpha)
   L
   \leq
   \sigma(\overline{\alpha})
   +
   \Cr{const:LengthEstimate}
   \left(
   \sigma(\alpha)^2
   \left\|
    \left(
    \frac{u_x}{\sqrt{1+u_x^2}}
    \right)_x
   \right\|^2_{L^2(0,1)}
   +
   |
   \sigma'(\alpha)
   |^2
   L^2
   \right)^\gamma
   ,
 \end{equation}
 where $L$ is the length of the graph of $u$, namely
 \begin{equation*}
  L=\int_0^1\sqrt{1+u_x^2}\,dx.
 \end{equation*}
\end{theorem}

The length estimate \eqref{eq:3.Length_Estimate_by_GBE} tells that the
length of a grain boundary can be estimated quantitatively by its
curvature and the $\sigma'(\alpha)$ if the grain boundary is close to
the equilibrium state. Note that the coefficient of $L^2$ on the
right-hand side, $|\sigma'(\alpha)|^2$ is close to $0$ since
$\sigma'(\overline{\alpha})=0$ and
$|\alpha-\overline{\alpha}|<\delta$. Further note that if $\sigma$ is a
constant, then $\sigma'=0$, so we have
\begin{equation*}
 L - 1
  \leq
  \Cr{const:LengthEstimate}
  \left\|
   \left(
    \frac{u_x}{\sqrt{1+u_x^2}}
   \right)_x
  \right\|^{2\gamma}_{L^2(0,1)};
\end{equation*}
thus, the length of the curve is estimated by the integral of its
curvature.

\begin{proof}[Proof of Theorem \ref{thm:3.LengthEstimate}]
 Since $\sigma'(\overline{\alpha})=0$, we can use the
 \L{}ojasiewicz-Simon inequality \eqref{eq:LS.LS_for_GBE}, and obtain
 \begin{multline*}
  |\sigma(\alpha)|^2
  \left(
  \left\|
  \frac{u_x}{\sqrt{1+u_x^2}}
  \right\|_{L^2(0,1)}^2
  +
  \left\|
  \left(
  \frac{u_x}{\sqrt{1+u_x^2}}
  \right)_x
  \right\|_{L^2(0,1)}^2
  \right)
  \\
  +
  |\sigma'(\alpha)|^2
  \left(
  \int_0^1\sqrt{1+u_x^2(x)}\,dx\right)^2
  \geq
  \Cr{const:LS_GBE}^2|E[u,\alpha]-E[0,\overline{\alpha}]|^{2(1-\theta)}.
 \end{multline*}
 Note that
 \eqref{eq:leq3.Periodic_PoincareInequality} with $f=\sqrt{1+u_x^2}$
 that
 \begin{equation*}
  \left\|
   \frac{u_x}{\sqrt{1+u_x^2}}
  \right\|^2_{L^2(0,1)}
  \leq
  \left\|
   \left(
    \frac{u_x}{\sqrt{1+u_x^2}}
   \right)_x
  \right\|^2_{L^2(0,1)}.
 \end{equation*}
 Then, taking $\gamma=\frac{1}{2(1-\theta)}$ and using the defintion of
 $L$, we have
 \begin{equation*}
  E[u,\alpha]-E[0,\overline{\alpha}]
   \leq
   \frac{1}{\Cr{const:LS_GBE}^{1-\theta}}
   \left(
    2|\sigma(\alpha)|^2
    \left\|
     \left(
      \frac{u_x}{\sqrt{1+u_x^2}}
     \right)_x
    \right\|_{L^2(0,1)}^2
    +
    |\sigma'(\alpha)|^2
    L^2
   \right)^{\gamma}
 \end{equation*}
 Then we have \eqref{eq:3.Length_Estimate_by_GBE} by applying
 $E[u,\alpha]=\sigma(\alpha)L$ and
 $E[0,\overline{\alpha}]=\sigma(\overline{\alpha})$, where
 $\Cr{const:LengthEstimate}=\frac{2^\gamma}{\Cr{const:LS_GBE}^{1-\theta}}$.
\end{proof}

\section*{Acknowledgments}

The authors would like to thank the anonymous referees for their helpful
comments and suggestions. The first author is supported by JSPS KAKENHI
grants JP22K03376 and JP23H00085. The third author is supported by JSPS
KAKENHI grants JP23K03180, JP23H00085, JP24K00531, and JP25K00918.


\end{document}